\documentclass[12pt]{article}
\usepackage{graphicx}
\DeclareGraphicsExtensions{.pdf,.jpg}
\usepackage{amsmath}
\usepackage{amssymb}
\usepackage{mathtools}
\usepackage{amsthm}
\usepackage{xcolor}

\newtheorem{theorem}{Theorem}
\newtheorem{definition}{Definition} 
\newtheorem{remark}{Remark}
\newtheorem{prop}{Proposition}
\newtheorem{lemma}{Lemma}

\def\R{\mathbb R}
\def\cFE{{\cal F}_{\rm E}}
\def\cF{{\cal F}_{\rm BE}}
\def\cGbe{{\cal G}_{\rm BE}}
\def\cGE{{\cal G}_{\rm E}}
\def\cgE{{\rm g}_{\rm E}}
\def\cgBE{{\rm g}_{\rm BE}}

\def\cN{{\cal N}}
\def\cR{{\cal R}}
\def\tR{\tilde{R}}
\def\tv{\tilde{v}}
\def\tg{\tilde{g}}

\def\tM{\tilde{M}}
\def\tPi{\tilde{\Pi}}
\def\tcF{\tilde{\cal F}_{\rm BE}}
\def\tcFE{\tilde{\cal F}_{\rm E}}
\def\hRbe{\hat{R}_{\rm BE}}
\def\hRE{\hat{R}_{\rm E}}
\def\eps{\epsilon}
\def\vep{\epsilon}
\def\beq{\begin{equation}}
\def\eeq{\end{equation}}

\begin{document}

\title{Asymptotic Behaviour of Time Stepping Methods for Phase Field Models} 

\author{
Xinyu Cheng 
\thanks{xycheng@math.ubc.ca, Department of
Mathematics, University of British Columbia, Vancouver, B.C.
Canada V6T 1Z2}
\and Dong Li
\thanks{madli@ust.hk, Department of
Mathematics, Hong Kong University of Science and Technology, 
Clearwater Bay, Kowloon, Hong Kong}
\and Keith Promislow
\thanks{kpromisl@math.msu.edu, 
Department of Mathematics, 
Michigan State, East Lansing, 48864 USA}
\and Brian Wetton
\thanks{wetton@math.ubc.ca, Department of
Mathematics, University of British Columbia, Vancouver, B.C.
Canada V6T 1Z2}
}

\maketitle

\begin{abstract}
Adaptive time stepping methods for metastable dynamics of the Allen Cahn and Cahn Hilliard equations are investigated in the spatially continuous, semi-discrete setting. 
We analyse the performance of a number of first and second order methods, formally predicting step sizes required to satisfy specified local truncation error $\sigma$ in the limit of small order parameter $\epsilon \rightarrow 0$ during meta-stable dynamics. The formal predictions are made under stability assumptions that include the preservation of the asymptotic structure of the diffuse interface, a concept we call profile fidelity.
In this setting, definite statements about the relative behaviour of time stepping methods can be made. 
Some methods, including all so-called energy stable methods but also some fully implicit methods, require asymptotically more time steps than others. 
The formal analysis is confirmed in computational studies.
We observe that some provably energy stable methods popular in the literature perform worse than some more standard schemes.  
We show further that when Backward Euler is applied to meta-stable Allen Cahn dynamics, the energy decay and profile fidelity properties for these discretizations are preserved for much larger time steps than previous analysis would suggest. The results are established asymptotically for general interfaces, with a rigorous proof for radial interfaces. It is shown analytically and computationally that for most reaction terms, Eyre type time stepping performs asymptotically worse due to loss of profile fidelity.
\end{abstract}


\section{Introduction}
\label{s:introduction}

The mathematical literature for computational methods for Allen-Cahn (AC) dynamics \cite{AC}, and its higher order relative Cahn Hilliard (CH) dynamics \cite{CH}, is dominated by the proposal, use, and analysis of so-called energy stable schemes \cite{energy1,shen1,wise10,qiao1}. AC and CH dynamics are gradient flows on an energy functional, and the solution should decrease that energy in time. Energy stable schemes guarantee that decrease no matter what time step is chosen. This is a desirable property not shared by standard fully implicit or semi-implicit (IMEX) time stepping methods. 
We will show in this work that some (but not all) fully implicit methods can outperform energy stable schemes when subject to fixed accuracy requirements. 
The recent article \cite{Xu} gives especially clear evidence that when time steps are chosen appropriately, fully implicit methods are conditionally energy stable, and further that the large time steps allowed by energy stable schemes can come at the cost of significant loss of accuracy. 
We extend the commentary in \cite{Xu} to show that in the metastable dynamic regime of AC and CH, some fully implicit methods can take optimally sized time steps. By optimal, we mean the asymptotically largest time steps as the order parameter $\epsilon \rightarrow 0$ that satisfy a given local error tolerance. 
Here, $\epsilon$ represents the width of interfacial layers in metastable dynamics and, like the authors of \cite{Xu}, we use the form of the equations scaled so that these dynamics transpire in an $O(1)$ time scale. When the dynamics are in this metastable regime, which dominates the 
time of typical phenomena of interest, definite statements about the behaviour of different time stepping methods can be made. This criteria does not take into account solver efficiency. However, we can make definite statements on how efficient solvers for nonlinear implicit time stepping need to be to outperform other methods.  

A combination of asymptotic analysis and careful computational work backs up our claims. In addition, we present a rigorous result for implicit time stepping for meta-stable AC dynamics in radial geometry that shows that asymptotically larger time steps can be taken than previous analysis would suggest. These time steps preserve the diffuse interface structure (a property that we call proflie fidelity) and also the energy decay property of the equations. This result is shown for a class of reaction terms. An interesting result in Section~\ref{s:Eyre_Type} shows that Eyre-type time stepping can perform asymptotically worse with most reaction terms, while implicit time stepping has uniform asymptotic behaviour over a class of reaction terms. This was predicted by the analysis and confirmed computationally. 

Our study focuses on pure materials science applications rather than the use of Cahn-Hilliard equations to track interfaces in so-called {\em diffuse interface methods} \cite{diffuse1} in which the CH dynamics are coupled to other physics. We consider the simplest form of AC and CH dynamics, whose Gamma limit (as $\epsilon \rightarrow 0$) is well understood and use that well known structure to gain insight into the behaviour of the schemes. The authors believe that the insight gained from these studies will also apply to schemes used for other materials science models which are less well understood. 

We consider a number of first and second order time stepping schemes: the energy stable Eyre's method \cite{eyre}; Backward (Implicit) Euler (BE) \cite{hairer}; Trapezoidal Rule (TR) \cite{hairer}; Second order Backward Differentiation Formula (BDF2) \cite{hairer}; Secant \cite{secant}; standard semi-implicit (linear IMEX) methods of first and second order \cite{imex}; first and second order Scalar Auxiliary Methods (SAV) \cite{SAV} for which a modified energy stability can be proved; and finally a second order Singular Diagonally Implicit Runge Kutta method with good stability properties (DIRK2) \cite{hairer}. The resulting implicit systems are considered in the spatially continuous semi-discrete setting in a 2D periodic domain, with numerical validation done with a suitably refined Fourier spectral approximation. Time step schemes that result in nonlinear systems are solved with Newton's iterations using the Preconditioned Conjugate Gradient Solver (PCG) developed in \cite{energy} at each iteration. Adaptive time stepping is done based on a user-specified local error tolerance $\sigma$. The variation of the number of time steps with $\epsilon$ for fixed $\sigma$ is predicted based on formal consideration of the local truncation error of the schemes in the metastable dynamics. The formal predictions are then validated in computational studies. With this criteria, first order BE performs better (asymptotically fewer time steps as $\epsilon \rightarrow 0$) than Eyre and first order IMEX and SAV. Second order TR and BDF2 perform better than Secant, DIRK2, and second order IMEX and SAV. The difference in both cases is asymptotically larger for CH than AC. These comparisons are also valid for computational time, using PCG counts as the measure, to similar accuracy. It is seen that optimal numbers of time steps are obtained when the dominant local truncation error is a higher order time derivative. This observation may have application in other systems with metastable dynamics. We observe that standard IMEX methods perform almost identically to SAV methods of the same order in the scenario we consider, at reduced computational cost. 

It is observed that the global accuracy of BE is better than a na\"ive prediction based on the size of the local truncation error would suggest. A formal analysis of the scheme for the AC case shows that the dominant error made in one time step is asymptotically smaller than expected. This is due to a special structure of the local truncation error for BE, in which the asymptotically largest term lies in a strongly damped space. 

We introduce the equations and numerical schemes in Section~\ref{s:equations} with some introductory analysis. The scaling for AC and CH is chosen so that the metastable interface dynamics (approximate curvature motion for AC and Mullins-Sekerka flow \cite{mullins} for CH) occurs in $O(1)$ time. In Section~\ref{s:meta} we examine the metastable dynamics of the equations and make predictions for the behaviour of the time steps with $\epsilon$ and local error tolerance $\sigma$ under stability assumptions which are verified numerically in Section~\ref{s:comp}. We give an asymptotic analysis for the surprising accuracy and stability properties for BE with large time steps applied to AC in Section~\ref{s:accuracy}. In Section~\ref{s:radial_main} we present the rigorous result for BE applied to AC with large time steps and also show the loss of profile fidelity for Eyre-type time stepping for most reaction terms. We end with a short discussion. 

\section{Equations and Schemes}  
\label{s:equations} 

We consider the simplest form of the AC dynamics for $u({\bf x},t)$ given by 
\begin{equation}
\label{eq:AC}
u_t = \Delta u - \frac{1}{\epsilon^2} f(u)
\end{equation}
where $f(u) = u^3-u$ is the classical form of the reaction term. More general reaction terms are considered in Section~\ref{s:radial_main}. CH dynamics is described by a higher order partial differential equation 
\begin{equation}
\label{eq:CH}
u_t = -\epsilon \Delta \Delta u + \frac{1}{\epsilon} \Delta f(u).
\end{equation}
For computational simplicity, we consider the two-dimensional (2D) cases of these equations in a doubly periodic cell $[0, 2 \pi]^2$. 
The time scaling in the equations above is chosen to give sharp interface (as $\epsilon \rightarrow 0$) motion in $O(1)$ time.  The sharp interface limit yields curvature driven flow for AC and a nonlocal Mullins-Sekerka flow for CH \cite{mullins}.  
Both types of dynamics have an associated energy functional 
\begin{equation}
\label{eq:energy}
{\mathcal E} = \int \left( |\nabla u|^2/2 + W(u)/\epsilon^2 \right)
\end{equation}
where $W(u) = \frac{1}{4} (u^2 -1)^2$ and the reaction term $f(u) = W^\prime (u)$. The energy $\mathcal E (t)$ is monotonic decreasing due to the gradient flow nature of the dynamics. For AC the gradient is in $L_2$ and for CH it is $H^{-1}$. 

\subsection{Time stepping} 
\label{s:time}

\subsubsection{Backward Euler}

We consider the simplest implicit scheme, first order Backward Euler (BE), also known as Implicit Euler. Applied to (\ref{eq:AC}) keeping space continuous, we have 
\[
\frac{u_{n+1}-u_n}{k_n}=\Delta u_{n+1}-\frac{1}{\epsilon^2} f(u_{n+1})\ .
\]
where $u_{n} ({\bf x})$ approximates the exact solution $u({\bf x}, t_n)$ and $k_n = t_{n+1} - t_n$ is the time step. We use the classical $f(u) = u^3 -u$ as mentioned above. Dropping the subscript on the time step and the unknown solution at time level $n+1$ we have the nonlinear problem 
\begin{equation}
\label{eq:BE}
u - k \Delta u + \frac{k}{\epsilon^2} f(u) = u_n
\end{equation}
for $u$ given $u_n$. 

\begin{definition} A time stepping scheme is said to have the {\em energy decay property} if ${\mathcal E} (u_{n+1}) \leq {\mathcal E} (u_n)$.
\end{definition}
This property could be conditional on the choice of time step size. Additionally, it could depend on $u_n$. If a scheme has the energy decay property for any  $u_n$ and $k$, the scheme is called {\em unconditionally energy stable}.  

\begin{theorem} Consider  \eqref{eq:BE}, assume that $u_n\in H^2(\Omega)$ and $u_n$ takes values in $[-1,1]$, then there exists $u \in H^2(\Omega)$ that solves \eqref{eq:BE} with values in $[-1,1]$. 
Define $f_\infty:=\max\{|f'(s)|,\, s\in[-1,1]\}$, then if $k\leq  2 \epsilon^2/f_\infty$ the solution $u$ is unique and satisfies the energy decay property. Note that the energy stability result was established earlier in \cite{Xu} with a different proof. 
\end{theorem} 
\begin{proof}
The existence of $u$ follows from the standard method of sub-/super-solutions applied to comparison functions $-1$ and $+1$. To establish uniqueness,
we assume $u_1$ and $u_2$ are solutions. Then their difference $w=u_1-u_2$ is a solution of
\begin{equation*}
\begin{aligned}
(1-k\Delta )w=-k\cdot \frac{f(u_1)-f(u_2)}{\epsilon^2}=-\frac{k}{\epsilon^2}\cdot f'(s(x)) w\ ,
\end{aligned}
\end{equation*}
where $s$ takes values between $u_1$ and $u_2$, and hence in $[-1,1]$.
Isolating $w$ leads to the elliptic problem
\begin{equation*}
\begin{aligned}
\left[1+\frac{k f'(s)}{\epsilon^2}-k\Delta \right]w=0,
\end{aligned}
\end{equation*}
and if $k < \epsilon^2/f_\infty $ then the corresponding elliptic operator is strictly positive and $w$ is zero by the maximum principle. 
To establish energy decay, we take the inner product of (\ref{eq:BE}) with the test function $u-u_n$:
\begin{equation*}
\begin{aligned}
\frac{1}{k}\int |u-u_n|^2+\frac12\int\left(|\nabla u|^2-|\nabla u_n|^2+|\nabla u-\nabla u_n|^2\right)=
-\frac{1}{\epsilon^2}\left(f(u),(u-u_n)\right)\ .
\end{aligned}
\end{equation*}
From the Fundamental Theorem of Calculus we develop the expansion,
\begin{equation*}
|F(u)-F(u_n) - f(u)(u-u_n)| =  \left|\int_u^{u_n}\ f'(s)(s-u_n)\ ds\right|\leq  \frac{f_\infty }{2}(u-u_n)^2. \\
\end{equation*}
Using this relation to eliminate  $f(u)$ yields the equality,
\begin{equation*}
(\frac1k-\frac{f_\infty}{2\epsilon^2}) \int |u-u_n|^2+ E[u]-E[u_n] \leq 0,
\end{equation*}
which yields energy decay for $k<2\epsilon^2/f_\infty$.
The Theorem is also true when homogeneous Neumann boundary conditions are specified.
\end{proof}
Thus we have existence of solutions to (\ref{eq:BE}) for any time step size, and uniqueness and energy stability under the resitriction $k \leq 2 \epsilon^2/f_\infty$. This is true for any $u_n$ under the restrictions of the Theorem. We shall see in Section~\ref{s:radial_main} that asympoticaly larger time steps $k = o(\epsilon)$ can be taken when the dynamics are slow (interface motion) with locally unique, energy stable solutions. This is verified in computational tests.  

\subsubsection{Eyre's Method} 

An alternative first order scheme to fully implicit BE was proposed by Eyre \cite{eyre}:
\begin{equation}
\label{eq:eyreAC}
u - k \Delta u + \frac{k}{\epsilon^2} u^3 = u_n+ \frac{k}{\epsilon^2} u_n 
\end{equation}
The scheme is derived conceptually by keeping a convex part of the reaction term $f(u) = u^3 - u$ implicit and a concave part explicit. In this sense, it is an IMEX method but an unusual one since a nonlinear term is kept implicit and a linear term is handled explicitly. The method has appealing properties:
\begin{theorem}[from \cite{eyre}] 
The time step (\ref{eq:eyreAC}) has a unique solution $u$ for any $u_n$ and $k$ that is unconditionally energy stable. 
\end{theorem}
Additional first order schemes considered are the SAV scheme \cite{SAV} and a linear IMEX method \cite{imex}:
\begin{equation}
\label{eq:IMEX1}
u - k \Delta u + \frac{Mk}{\epsilon^2} u = u_n- \frac{k}{\epsilon^2} \left( u_n^3-(M+1)u_n \right) 
\end{equation}
with $M>0$, sometimes called a stabilization term. We take $M=2$ since that makes the left hand side a linearization about the far field values, but computational performance is relatively insensitive to $M$. 
The SAV scheme is energy stable with a modified energy. We use the same stabilization coefficient as above in the SAV scheme. There is a class of linearly implicit energy stable schemes \cite{xinyu,LI1,LI2} that require an asymptotically large stabilization term $O(\epsilon^{-p})$ with $p$ large and increasing from AC to CH and 2D to 3D for the analysis. These methods are theoretically interesting but are extremely inaccurate and not useful for practical applications. We have further discussion of these schemes in Remark~\ref{rem:linEstab}. 

All time stepping schemes can be applied to CH (\ref{eq:CH}), with BE and Eyre shown below: 
\begin{align*}
u + k \Delta \Delta u - \frac{k}{\epsilon^2} \Delta f(u) & = u_n & \mbox{BE} \\
u + k \Delta \Delta u - \frac{k}{\epsilon^2} \Delta u^3 & = u_n - \frac{k}{\epsilon^2} \Delta u_n & \mbox{Eyre} 
\end{align*}
In this case, BE is known to have unique solutions with the energy decay property when $k< \epsilon^3$ \cite{Xu} and Eyre is unconditionally energy stable \cite{eyre}.  

\subsubsection{Second Order Schemes} 

We also consider the second order methods Trapezoidal Rule (TR), Secant (S) \cite{secant}, Second Order Backward Differencing (BDF2), and Second Order Singular Diagonal Implicit Runge Kutta (DIRK2) \cite{hairer} methods. These are described below for $u_t = \mathcal{F}(u)$ with 
\begin{align*}
 & \mathcal{F}(u) = \Delta u - f(u)/\epsilon^2 & \mbox{for AC} \\ 
\mbox{and\ \ } & \mathcal{F}(u)  = -\epsilon \Delta \Delta u + \Delta f(u)/\epsilon & \mbox{for CH} 
\end{align*}
With this notation:
\begin{align*}
\mbox{(TR) \ \ \ } & u- \frac{k}{2} \mathcal{F}(u) u = u_n + \frac{k}{2} \mathcal{F}(u_n) \\
\mbox{(BDF2) \ \ \ } & \frac{3u}{2} - k \mathcal{F}(u)  = 2 u_n -\frac{1}{2} u_{n-1}. 
\end{align*}
Secant is a variant of TR with the term $f(u) - f(u_n)$ replaced by 
\[
(W(u) - W(u_n))/(u-u_n)
\] 
where $W$ is the energy term from (\ref{eq:energy}). It is known to be conditionally energy stable \cite{secant}. For the simple form of $W$ we have taken, the expression above can be factored explicitly. 
DIRK2 is a two stage method 
\begin{align*}
& u_{*} -\alpha k \mathcal{F}(u_*) = u^n \\
& u - \alpha k \mathcal{F}(u)  = u^n + (1-\alpha) k \mathcal{F}(u_*)
\end{align*} 
with $\alpha=1-1/\sqrt{2}$. Both DIRK2 and BDF2 are A-stable, and so preferable to TR and Secant from the perspective of stiff ODE solver theory \cite{hairer}. A second order linear IMEX method (SBDF2 \cite{imex}) and two variants of second order SAV methods based on BDF2 are also considered. 

\subsection{Spatial discretization and solution procedure}
\label{s:space}

The current work concentrates on the time stepping errors, and it is convenient to consider the semi-discrete, spatially continuous approximation. 
This idealization is approximated well by the Fourier spectral spatial discretization.
 The computational results shown have sufficient spatial resolution that spatial errors do not affect the results in the digits shown. The computations are done in a full 2D setting, rather than in a reduced dimensional radial setting as could be done, in order to give PCG iteration counts for the nonlinear time stepping methods that have meaning for more general computations. Note that the PCG counts are independent of spatial resolution when the problem is resolved. 

\subsection{Error estimation and adaptive time stepping} 
\label{s:error} 

We perform two time steps of the same size $k$ in order to use a specialized predictor $u_p$ for $u_{n+2}$. 
\begin{equation}
\label{eq:estimator}
u_p= u_n + \frac{k}{3} \left(\mathcal{F}(u_n) + 4 \mathcal{F}(u_{n+1}) + \mathcal{F}(u_{n+2}) \right)
\end{equation}
where $\mathcal{F}(u) = \Delta u - f(u)/\epsilon^2$ for AC and $-\epsilon \Delta \Delta u + \Delta f(u)/\epsilon$ for CH as above. Time step sizes are adjusted so that 
\[
\| u_{n+2} - u_p \|_\infty \leq \sigma. 
\]
The predictor $u_p$ is formally one order more accurate than the numerical approximation $u_{n+2}$ from time stepping, up to fifth order. The predictor has an inherent  dominant  local error $ k^5 u_{ttttt}/90$ that is a pure time derivative of $u$, which is shown below in Section~\ref{s:local} to be a desirable property. 

For the one step methods, the time step is adjusted adaptively to maintain a local error below $\sigma$ as described in \cite{energy}. For BDF2 and its linear variants, time steps are only adjusted by a factor of two. When time steps are reduced (using Hermite cubic interpolation for the restart value) or increased, four time steps are taken before checking the local error to allow relaxation of the initial error layer.  

\section{Local Truncation Errors in Metastable Dynamics}
\label{s:meta}

\subsection{Metastable dynamics} 

In our formulation, it is known that after a short time $O(\epsilon^2)$ solutions to AC tend to interfaces between regions of solution near the equilibrium values, 
\[
u\approx\pm 1. 
\]
These interfaces have width $\epsilon$ and move approximately with curvature motion. We refer to this dynamics as metastable or slow, even though with the particular time scaling we have chosen it occurs in  in $O(1)$ time. For the majority of the time, the solution will be in this regime, so we concentrate now on the expected and observed behaviour of time stepping in this setting. 
With the choice of $f(u) = u^3 - u$, we have 
\begin{equation}
\label{eq:asy}
u(x,t) \approx g(z) 
\end{equation}
with $g(z) := \tanh (z/\sqrt{2})$ and $z = {\rm dist} (x, \Gamma)/\epsilon$, 
where $\Gamma$ is the approximate interface with arc length parameter $s$ moving with curvature motion (normal velocity equal to curvature). We fix its location at the $u=0$ level set. The local coordinates $(s,z)$ are shown in 
Figure~\ref{fig:local}. This structural result on the metastable solution can be obtained with formal asymptotics. In the outer asymptotic region for AC the solution takes the form $u = \pm 1$ to all orders. Curvature motion as the limit $\epsilon \rightarrow 0$ has been proven rigorously \cite{pego2,ABC94}. 

CH has the same metastable solution structure (\ref{eq:asy}) with normal interface velocity given by Mullins-Sekerka flow, in $O(1)$ time in our scaling (\ref{eq:CH}). We refer the reader to the review article \cite{Savin}  for details. 

\begin{figure}
\centerline{
\includegraphics[width=8cm]{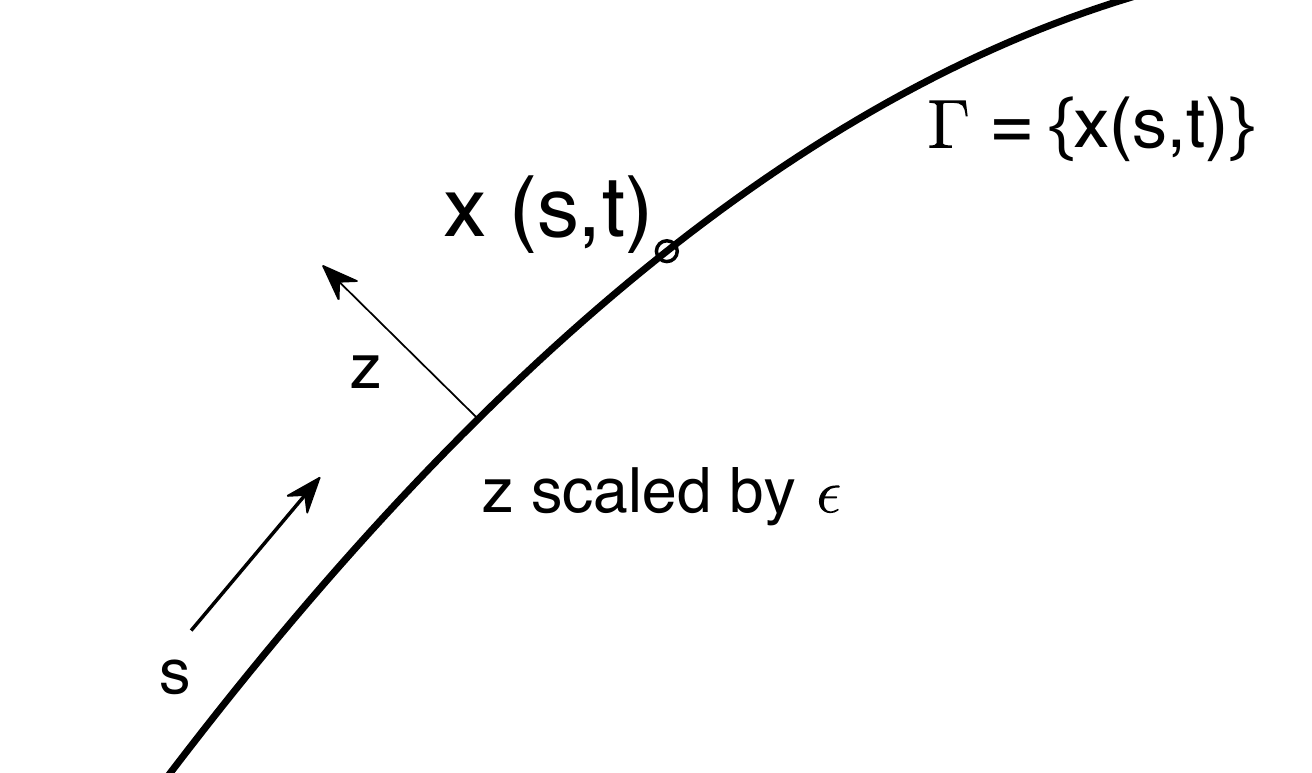}
}
\caption{Sketch of the local coordinates of the metastable solution 
\label{fig:local}}
\end{figure}

From (\ref{eq:asy}), we see that time and space derivatives are large near the interface. Starting with 
\[
u(x,t) \approx g({\rm dist} (x,\Gamma)/\epsilon)
\]
we can take a time derivative to obtain:
\[
u_t \approx g^\prime({\rm dist} (x,\Gamma)) V/\epsilon
\]
where $V$ is the normal velocity at the point on $\Gamma$ closest to $x$. Formally taking higher derivatives in this pattern yields: 
\begin{equation}
\label{eq:timed}
\frac{\partial^n u}{\partial t^n} = O(\epsilon^{-n}).
\end{equation}
This is used to analyze the truncation error of the time stepping schemes. 

\subsubsection{Predicted time step sizes for AC} 
\label{s:predict} 

A standard strategy for adaptive time stepping is to have a user specified local error tolerance of $\sigma$. The error for each time step is estimated and the time step adjusted so that there is an estimated error in that single time step less than $\sigma$. It is known that the dominant local truncation error for BE is $k^2 u_{tt}/2$ which 
in metastable dynamics is $O(k^2/\epsilon^2)$ from (\ref{eq:timed}).
The local truncation error restriction then requires time steps of size
\[
k = O(\sqrt{\sigma} \epsilon) \mbox{\ \ \ (BE)} 
\]

We now proceed to determine the expected behaviour of time steps with $\epsilon$ and $\sigma$ from the other schemes. We can write the BE scheme (\ref{eq:BE}) and Eyre's scheme (\ref{eq:eyreAC}) for AC in an instructive way 
\begin{eqnarray*}
u - u^n - k \Delta u + k \left[ u^3 - u \right]/\epsilon^2 & = & 0 \mbox{\ (BE)} \\
u - u^n - k \Delta u + k \left[ u^3 - u \right]/\epsilon^2 + k(u - u_n)/ \epsilon^2 & = & 0 \mbox{\ (Eyre).}
\end{eqnarray*}
Knowing that the truncation error for BE is $O(k^2/\epsilon^2)$ we see that the truncation error for the Eyre scheme is dominated by the last term in its expression above, which has leading order $k^2u_t/\epsilon^2 = O(k^2/\epsilon^3)$. Our time step prediction in this case is 
\[
k = O(\sqrt{\sigma} \epsilon^{3/2}) \mbox{\ \ \ (Eyre)} 
\]
Thus, the advantage of the Eyre scheme to be able to take large time steps and remain energy stable is never realized if accurate computational results are required. Reference \cite{Xu} has an alternate way to view the loss of accuracy that does not highlight this asymptotic difference. The 
first order IMEX and SAV schemes have the same asymptotic behaviour as Eyre. 

\begin{remark}
\label{rem:profstab}
The formal local error analysis above relies on the stability of the schemes in metastable dynamics under the resulting time step restrictions. More than simple stability, the analysis requires that the time stepping preserves the asymptotic structure of the diffuse interface. This is the concept we have named {\em profile fidelity}. All predicitions described in this section lead to time stepping that preserves profile fidelity for the classical choice of $f(u) = u^3 - u$. We observe the predicted time step behaviour in $\epsilon$ and $\sigma$ computationally. In Section~\ref{s:Eyre_Type} we show that for (most) other reaction terms, Eyre time stepping loses profile fidelity for time steps $k = O(\epsilon^{3/2})$ and in these cases,  $k = O(\epsilon^2)$ is needed for accuracy. 
\end{remark}

\begin{remark}
\label{rem:linEstab}
The first order, linearly implicit energy stable scheme for 2D AC is analyzed in \cite{xinyu}. The analysis requires a stabilization term of order $\epsilon^{-2} | \ln \epsilon|$. If such a scheme were implemented, the time steps required for a local error tolerance of $\sigma$ would be $k = O(\sqrt{\sigma} \epsilon^{5/2} / | \ln \epsilon|)$, prohibitively small for practical computation. 
\end{remark}

We can determine the dominant term in the local truncation errors of the second order schemes applied to AC:
\begin{align*}
\mbox{(TR)\ \ \ } & k^3 u_{ttt}/12 = O(k^3/\epsilon^3) \\
\mbox{(S) \ \ \ } & k^3 \left(u_{ttt}/12 + u u_{t}^2 /(2 \epsilon^2) \right) = O(k^3/\epsilon^4) \\
\mbox{(DIRK2) \ \ \ } & k^3 \left( (\alpha^2(1-\alpha) + \alpha/2 -1/6) u_{ttt} - 
	3 \alpha^2(1-\alpha) u u_t^2/(2 \epsilon^2) \right)= O(k^3/\epsilon^4) \\
\mbox{(BDF2) \ \ \ } &  - k^3 u_{ttt}/3 = O(k^3/\epsilon^3) \\
\mbox{(SBDF2) \ \ \ } & k^3 \left(3 u^2 + (M+1) \right) u_{tt}/\epsilon^2 = O(k^3/\epsilon^4)
\end{align*}
We consider two second order SAV variants based on how an extrapolated approximation is computed. If the extrapolated value of $u^{n+1}$ is taken as $2u^n - u^{n-1}$ the scheme (referred to as SAV2-A) behaves similarly to SBDF2. If the extrapolated value is computed with a first order linear IMEX scheme as suggested in \cite{SAV} (referred to as SAV2-B), the scheme has a local truncation error of order $k^3/\epsilon^5$. 
The results are summarized in Table~\ref{t:AC}. It is clear that BE takes asymptotically (as $\epsilon \rightarrow 0$) fewer time steps than Eyre, although they are both first order in time step size. TR and BDF2 take asymptotically fewer time steps than Secant, DIRK2, SBDF2, SAV2-A and SAV2-B although they are all second order methods. The computations in Section~\ref{s:comp} below show that these time step estimates correspond to real computational behaviour. 

\begin{remark}
\label{rem:kprofile} 
We predict the number $M$ of time steps in Tables~\ref{t:AC} and~\ref{t:CH} and how it varies with $\epsilon$ and $\sigma$. As shown in Figure~\ref{fig:kprofile} we are also predicting how a profile of time steps $k(t)$ behaves with $\epsilon$ and $\sigma$. 
\end{remark}

\begin{table}
\centerline{
\begin{tabular}{|c|c|c|c|} \hline 
Method (AC) & $L$ & $k$ & $M = O(1/k)$  \\ \hline
BE & $k^2/\epsilon^2$ & $\sqrt{\sigma}\epsilon$ & $1/(\sqrt{\sigma}\epsilon)$ \\	
Eyre, IMEX1, SAV1 & $k^2/\epsilon^3$ & $\sqrt{\sigma}\epsilon^{3/2}$ & 
	$1/(\sqrt{\sigma}\epsilon^{3/2})$ \\
TR, BDF2 & $k^3/\epsilon^3$ & $\sqrt[3]{\sigma}\epsilon$ & 
	$1/(\sqrt[3]{\sigma}\epsilon)$  \\
S, DIRK2, SBDF2, SAV2-A &  $k^3/\epsilon^4$ & $\sqrt[3]{\sigma}\epsilon^{4/3}$ & 
	$1/(\sqrt[3]{\sigma}\epsilon^{4/3})$ \\ 
SAV2-B & $k^3/\epsilon^5$ & $\sqrt[3]{\sigma}\epsilon^{5/3}$ & 
	$1/(\sqrt[3]{\sigma}\epsilon^{5/3})$ \\ \hline 
\end{tabular}
}
\caption{Order predictions for the behaviour of the numerical schemes with local error tolerance $\sigma$ in the metastable regime of AC dynamics. Here, $L$ is the local error, $k$ is the time step size, and $M$ is the number of time steps to reach a fixed end time.  
\label{t:AC}
}
\end{table}

\subsubsection{Predicted time step sizes for CH}

The same local truncation analysis can be done for the CH in the metastable regime where the solution has the same interface structure (\ref{eq:asy}) with the interface $\Gamma$ moving approximately with Mullins-Sekerka flow in $O(1)$ time. BE, TR, BDF2, and SBDF2 have the same error expressions as above, but Eyre, Secant and DIRK2 have local truncation errors when applied to CH listed below:
\begin{align*}
\mbox{(Eyre) \ \ } & k^2 (u_{tt}/2 - \Delta u_t/\epsilon) = O(k^2/\epsilon^4) \\
\mbox{(S) \ \ \ } & k^3 \left(u_{ttt}/12 -\Delta (u u_{t}^2)/(2 \epsilon) \right) = O(k^3/\epsilon^5) \\
\mbox{(DIRK2) \ \ \ } & k^3 \left( (\alpha^2(1-\alpha) + \alpha/2 -1/6) u_{ttt} +
	3 \alpha^2(1-\alpha)\Delta 
	(u u_t^2)/(2 \epsilon) \right)= O(k^3/\epsilon^5) \\
\mbox{(SBDF2) \ \ \ } & k^3 \left(3 u^2 + (M+1) \right) \Delta u_{tt}/\epsilon = O(k^3/\epsilon^5)
\end{align*} 
where we have used the fact that the Laplacian $\Delta$ increases the size of terms by $1/\epsilon^2$ near the interface. The 
first order IMEX and SAV schemes have the same asymptotic behaviour as Eyre. 
SAV2-A behaves similarly to SBDF2 as before, with SAV2-B worse by a power of $\epsilon$ as for the AC case above. The results are summarized in Table~\ref{t:CH}. The predictions in this table are validated in the numerical experiments in the next section. Although the methods all have the formal order of accuracy in terms of time step size, the behaviour as $\epsilon \rightarrow 0$ varies significantly. Note that the gap between BE and the other first order schemes, and between TR/BDF2 and Secant/DIRK2/SBDF2/SAV2-A is wider for CH dynamics than it was for AC. 

\begin{table}
\centerline{
\begin{tabular}{|c|c|c|c|} \hline 
 Method (CH) & $L$ & $k$ & $M = O(1/k)$ \\ \hline
BE & $k^2/\epsilon^2$ & $\sqrt{\sigma}\epsilon$ & $1/(\sqrt{\sigma}\epsilon)$ \\	
Eyre, IMEX1, SAV1 & $k^2/\epsilon^4$ & $\sqrt{\sigma}\epsilon^2$ & 
	$1/(\sqrt{\sigma}\epsilon^2)$ \\
TR, BDF2 & $k^3/\epsilon^3$ & $\sqrt[3]{\sigma}\epsilon$ & 
	$1/(\sqrt[3]{\sigma}\epsilon)$ \\
S, DIRK2, SBDF2, SAV2-A &  $k^3/\epsilon^5$ & $\sqrt[3]{\sigma}\epsilon^{5/3}$ & 
	$1/(\sqrt[3]{\sigma}\epsilon^{5/3})$ \\ \hline
\end{tabular}
}
\caption{Order predictions for the behaviour of the numerical schemes with local error tolerance $\sigma$ in the metastable regime of CH dynamics. Here, $L$ is the local error, $k$ is the time step size, and $M$ is the number of time steps to reach a fixed end time.
\label{t:CH}
}
\end{table}

\subsubsection{Discussion: the source of increased local error} 
\label{s:local} 

In the metastable regime, the two terms in AC and CH (diffusion and nonlinear reaction) are both large but approximately cancel to give the slow dynamics. The methods with asymptotically (as $\epsilon \rightarrow 0$) small local errors (BE, TR, BDF2) have dominant truncation errors that are pure time derivatives of the solution, which inherit this high order cancellation. The other methods which have large local errors have truncation errors that involve the reaction term individually. This imbalance amplifies the size of the error. As an example, DIRK2 applied to $u_t = \mathcal{F}(u)$ has an error proportional to $ \mathcal{F}^{\prime \prime} u_t^2$. From this discussion, we believe the ranking of the schemes in this work will also apply to other nonlinear problems with metastable dynamics. 

\section{Computational Results} 
\label{s:comp}

\subsection{Allen Cahn} 
\label{s:ACcomp} 

We take initial conditions in the form of a radial front 
\[
\tanh \frac{\sqrt{(x-\pi)^2 + (y-\pi)^2}-2}{\epsilon \sqrt{2}}
\]
and compute with $\epsilon$ = 0.2, 0.1, 0,05 and 0.025. The benchmark for accuracy is the time $T$ at which the value at the domain centre $(\pi, \pi)$ changes from negative to positive. Except for the exponentially small (in $\epsilon$) derivative discontinuities at the periodic boundaries, the dynamics approximate the sharp interface limit of curvature motion of a circle. The expectation from asymptotic analysis of the sharp interface limit is that 
\[
T = 2 + O(\epsilon^2). 
\]
This is confirmed by the numerical solutions below. Some snapshots of the dynamics are shown in Figure~\ref{fig:AC}. A video of the dynamics is also available \cite{YTAC}.   

\begin{figure}
\centerline{
\includegraphics[width=5cm]{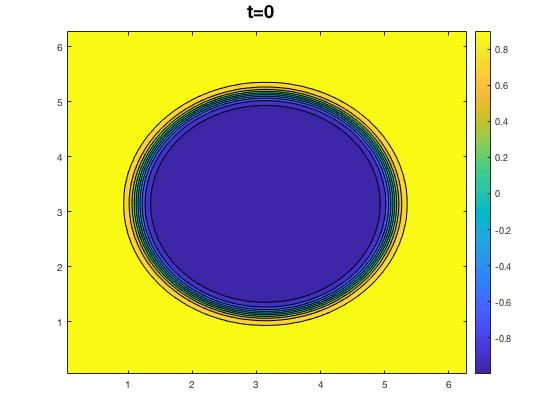}
\includegraphics[width=5cm]{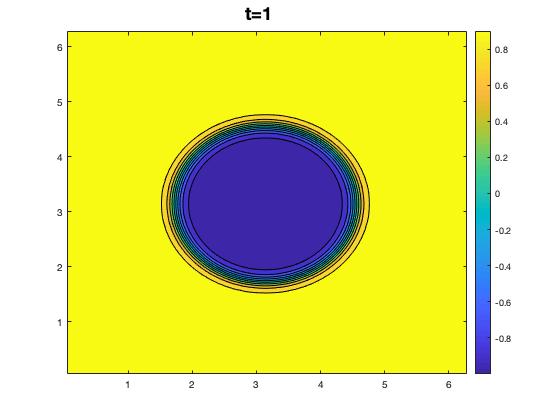}
\includegraphics[width=5cm]{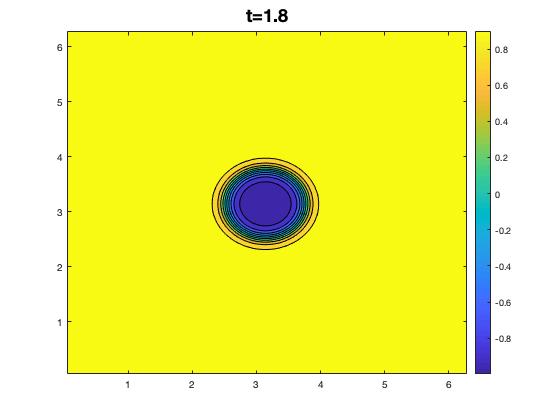}
}
\caption{Allen Cahn dynamics with $\epsilon = 0.1$.
\label{fig:AC}}
\end{figure} 

\subsubsection{First order methods} 

The PCG approach is known to have bounded condition number under the scaling $k = C\epsilon^2$ for BE with $C<1$ \cite{Xu} and we observe good behaviour in the example below even with $C>1$ in the metastable regime. It is observed computationally in this work that the PCG for Eyre's method is independent of $k$ and $\epsilon$ although the authors are not aware of a proof in the literature.  

Results of the numerical experiments in which $\sigma$ and $\epsilon$ were varied for BE and Eyre are shown in Tables~\ref{t:fixede} and~\ref{t:fixeds}. Spatial errors do not affect the digits shown in any of the computational results in this paper. 
\begin{table}
\centerline{ 
\begin{tabular}{|c||c|c|c||c|c|c|} \hline 
&\multicolumn{3}{|c||}{BE}&\multicolumn{3}{|c|}{Eyre}\\ \hline
$\sigma$ & $M$ & CG & $E$ & $M$ & CG & $E$ \\ \hline 
1e-4 & 717 & 5,348 [7.46] & 0.003 & 2,350 & 14,856 [6.32] & 0.047 \\ 
1e-5 & 2,225 (3.10) & 9,448 [4.24] & 0.001 & 7,351 (3.12) & 28,263 [3.85] & 0.014 \\
1e-6 & 7,010 (3.15) & 23,017 [3.28] & 0.001 & 23,172 (3.15) & 68,148 [2.94] & 0.004 \\ \hline
\end{tabular}
}
\caption{Computational results for the AC benchmark problem with fixed $\epsilon = 0.2$ and local error tolerance $\sigma$ varied. BE results are on the left, Eyre on the right. Here, $M$ is the total number of time steps taken (with the ratio to the value above in brackets), CG is the number of conjugate iterations (with the ratio to the number of time steps in brackets), $E$ is the error in the benchmark time. 
\label{t:fixede}
}
\end{table}

\begin{table}
\centerline{ 
\begin{tabular}{|c||c|c|c||c|c|c|} \hline 
&\multicolumn{3}{|c||}{BE}&\multicolumn{3}{|c|}{Eyre}\\ \hline
$\epsilon$ & $M$ & CG & $E$ & $M$ & CG & $E$ \\ \hline 
0.2 & 717 & 5,348 [7.46] & 0.003 & 2,350 & 14,856 [6.32] & 0.047 \\ 
0.1 & 1,291 (1.80) & 12,354 [9.57] & 0.001 & 6,463 (2.75) & 44,717 [6.92] & 0.069 \\
0.05 & 2,412 (1.87) & 27,782 [11.52] & 0.001 & 18,218 (2.83) & 143,416 [7.87] & 0.099 \\ 
0.025 & 4,630 (1.92) & 64,884 [14.01] & $*$ & 52,595 (2.89) & 497,846 [9.47] & 0.141 \\
\hline
\end{tabular}
}
\caption{Computational results for the AC benchmark problem with fixed local error tolerance $\sigma = 10^{-4}$ and $\epsilon$ varied. Here, $M$ is the total number of time steps taken (with the ratio to the value above in brackets) and CG is the number of preconditioned conjugate gradient iterations (with the ratio to the number of time steps in brackets), $E$ is the error in the benchmark time with $*$ denoting a result correct to three decimal places. 
\label{t:fixeds}
}
\end{table}

Table~\ref{t:fixede} validates the second order $O(k^2)$ local truncation error since the number of time steps was predicted to be $M = O(1/\sqrt{\sigma})$ for both methods with $\epsilon$ constant, noting that $\sqrt{10} \approx 3.16$. Such results for other schemes and for the CH benchmark problem below are not shown, but verify the formal accuracy of the schemes. Table~\ref{t:fixeds} validates the prediction of $M = O(1/\epsilon)$ for BE and $M = O(1/\epsilon^{3/2})$ for Eyre with $\sigma$ constant, noting that $2^{3/2} \approx 2.83$. 
Both tables validate the prediction that for the same local tolerance $\sigma$, Eyre involves more computational work than BE and gives less accurate answers. CG counts for both methods are small as expected. 
You see (unexpectedly) that the final accuracy of BE does not seem to degrade as $\epsilon \rightarrow 0$ for fixed $\sigma$. This is discussed in Section~\ref{s:accuracy} below. Although BE does not guarantee energy stability, no step accepted by the local error tolerance exhibited an energy increase. 

For completeness, we show the time step sizes as a function of time for BE in Figure~\ref{fig:kprofile} with $\epsilon$ and $\sigma$ varied.  As mentioned in Remark~\ref{rem:kprofile} our predictions for the behaviour of the time steps sizes $k$ as $\epsilon$ and $\sigma$ are varied describe a profile $k(t)$. 

We repeat the $\epsilon \rightarrow 0$ study for IMEX1 and SAV1 in Table~\ref{t:AC1}. These methods require a fixed number of FFT calculations per time step to invert the constant coefficient linear implicit aspect of the schemes, with SAV1 requiring four times as many solves as IMEX1. It is seen that IMEX1 behaves almost identically to SAV1 and both are superior to Eyre's method when computational cost is considered. In the context of this study, there is no benefit from the theoretical guarantees of energy stable schemes and BE is the optimal (with our asymptotic definition) first order scheme with IMEX1 the runner up. This will remain true for other nonlinear solver strategies for BE as long as they require fewer than $O(1/\sqrt{\epsilon})$ iterations when adaptive time steps are taken.  

\begin{table}
\centerline{ 
\begin{tabular}{|c||c|c||c|c|} \hline 
&\multicolumn{2}{|c||}{IMEX1}&\multicolumn{2}{|c|}{SAV1}\\ \hline
$\epsilon$ & $M$ &  $E$ & $M$ &  $E$ \\ \hline 
0.2 & 3,932 & 0.067 & 3,936 &  0.067 \\ 
0.1 & 11,110 (2.83) &  0.096 & 11,112 (2.82) & 0.096 \\
0.05 & 31,676 (2.85) &  0.138 &  31,682 (2.85) &  0.138 \\ 
0.025 & 90,748 (2.86) &  0.198 & 90,760 (2.86) &  0.198 \\
\hline
\end{tabular}
}
\caption{Computational results for the AC benchmark problem with fixed local error tolerance $\sigma = 10^{-4}$ and $\epsilon$ varied. Here, $M$ is the total number of time steps taken (with the ratio to the value above in brackets) and $E$ is the error in the benchmark time. 
\label{t:AC1}
}
\end{table}

\begin{figure}
\centerline{
\includegraphics[width=6cm]{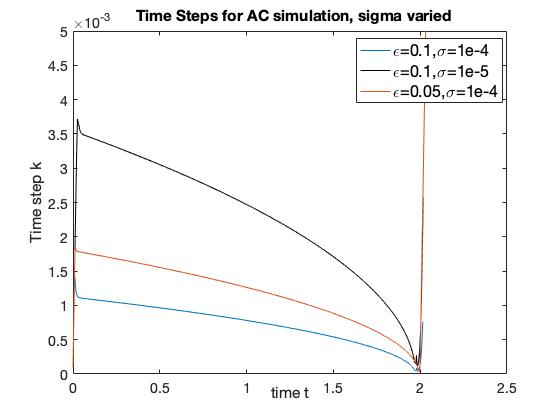}
}
\caption{Time steps $k$ for Allen Cahn dynamics with $\epsilon$ and $\sigma$ varied using BE. The time steps decrease in size as the simulation approaches the topological singularity at $t \approx 2$. Note that the profiles $k(t)$ have the same shape and are scaled in $\sigma$ and $\epsilon$ according to our theoretical predictions. 
\label{fig:kprofile} 
}
\end{figure} 

\begin{remark}
Note that for the BE computation for $\epsilon = 0.025$ we can still get reasonable accuracy taking $\sigma = 10^{-2}$. In this case, the maximum value of $k/\epsilon^2$ is 14.6. Clearly, the theory which guarantees existence of solutions and energy decay for $k < \epsilon^2$ \cite{Xu} can be improved for metastable dynamics. This is explored in the analysis in Sections~\ref{s:accuracy} and~\ref{s:radial_main} below. 
\end{remark}

\subsubsection{Second Order Methods}

The CG counts of all the nonlinear second order methods are relatively insensitive to $\epsilon$, similar to the first order methods shown above. We show the number of time steps used for the seven methods in Table~\ref{t:AC2}, for $\sigma = 10^{-4}$ fixed and $\epsilon$ varied. The superiority of TR and BDF2 is clearly seen with $M = O(1/\epsilon)$, compared to $M = O(1/\epsilon^{4/3})$ (noting that $2^{4/3} \approx 2.52$) for Secant, DIRK2, SBDF2, SAV2-A and  $M = O(1/\epsilon^{5/3})$ (noting that $2^{5/3} \approx 3.18$) for SAV-B as predicted above. The pattern in the number of time steps for the multi-step methods is a bit rougher due to the strict criteria we have used for adaptive time step change. As above, we see no benefit from the theoretical guarantees of energy stable schemes. Fully implicit methods TR and BDF2 are asymptotically optimal in terms of the number of time steps and are computationally optimal if the solvers require fewer than $O(1/\sqrt[3]{\epsilon})$ iterations when adaptive time steps are taken (which appears to be the case with the Newton PCG solver we used). SBDF2 is the runner up and notably it is comparable to the fully implicit DIRK2 method but does not have the overhead of a nonlinear solve. 

It is interesting to note that the slight change in the extrapolation procedure in the SAV2 schemes makes such a difference to their asymptotic performance. It is confirmation that merely considering the order of time stepping scheme and its theoretical energy stability properties is not the whole story. 

\begin{table}
\centerline{ \footnotesize
\begin{tabular}{|c|c|c|c|c|c|c|c|} \hline 
$\epsilon$ & TR & S & BDF2 & DIRK2 & SBDF2 & SAV2-A & SAV2-B \\ \hline 
0.2 & 170 & 236 & 280 & 180 & 588 & 768 & 1,572 \\
0.1 & 278 (1.64) & 512 (2.16) & 472 (1.69) & 364 (2.02) & 1,384 (2.35) & 1,572 (2.04) & 5,436 (3.46) \\
0.05 & 492 (1.77) & 1,208 (2.36) & 860 (1.82) & 814 (2.24) & 3,260 (2.35) & 3,392 (2.16) & 15,088 (2.78) \\ 
0.025 & 916 (1.86) & 2,960 (2.45) & 1,632 (1.90) & 1,894 (2.33) & 7,600 (2.33) & 7,980 (2.35) & 
48,048 (3.18) \\ \hline 
\end{tabular} 
}
\caption{
Computational results for the second order methods applied to the AC benchmark problem with fixed local error tolerance $\sigma = 10^{-4}$ and $\epsilon$ varied. Shown are the total number of time steps taken (with the ratio to the value above in brackets) 
\label{t:AC2} 
}
\end{table}

\subsection{Cahn Hilliard} 

For the initial conditions we take 
\[
\tanh \left( \frac{r-5/2}{\epsilon \sqrt{2}} \right) + 
	\tanh \left( \frac{3/2-r}{\epsilon \sqrt{2}} \right) +1 
\]
with $r = \sqrt{(x-\pi)^2 + (y-\pi)^2}$
and compute with $\epsilon$ = 0.2, 0.1, 0,05 and 0.025. As before, the benchmark is the time $T$ at which the value at the domain centre $(\pi, \pi)$ changes from negative to positive. The dynamics are shown in Figure~\ref{fig:CH} and a video of the dynamics is also available \cite{YTCH7} .   

\begin{figure}
\centerline{
\includegraphics[width=5cm]{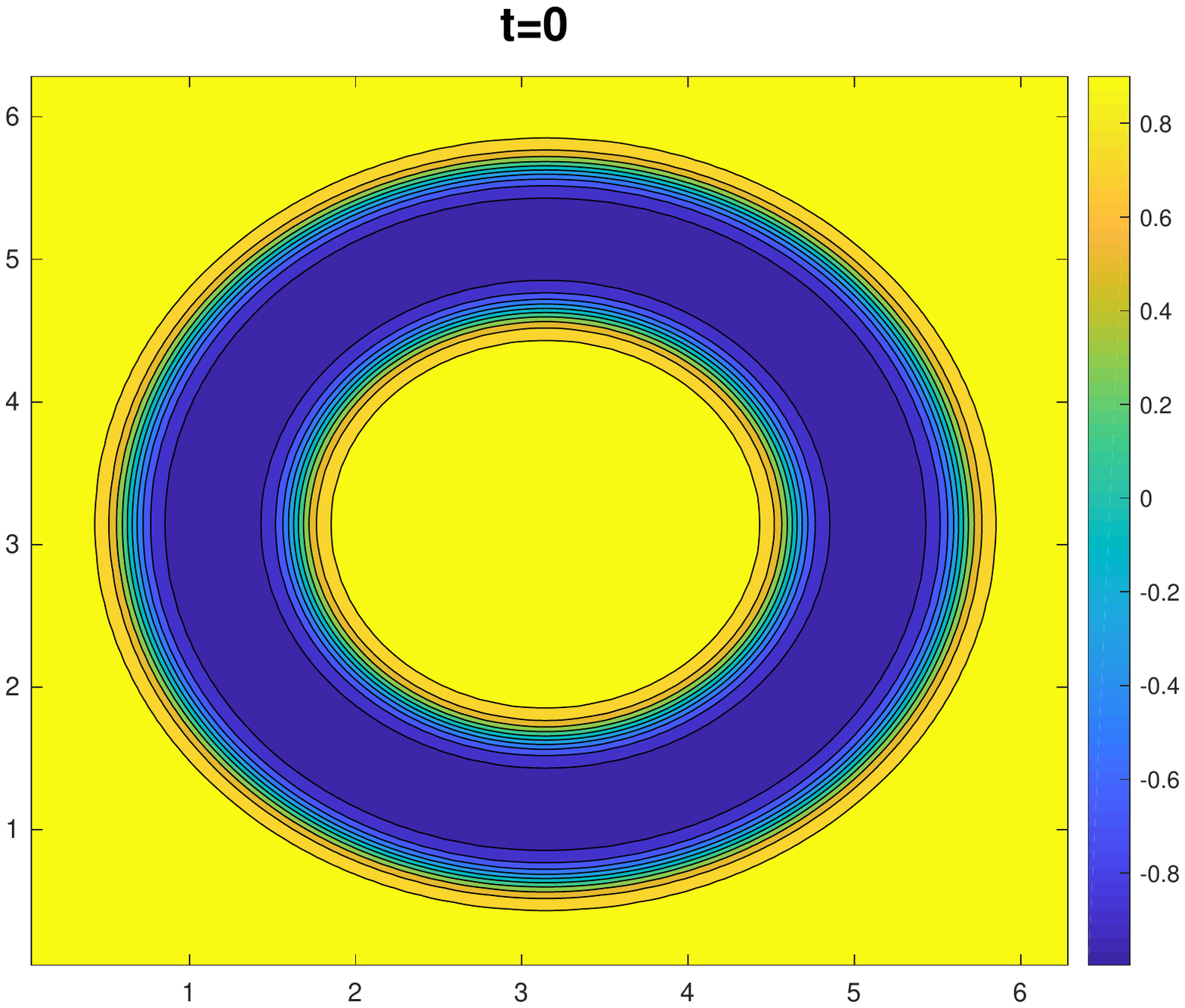}
\includegraphics[width=5cm]{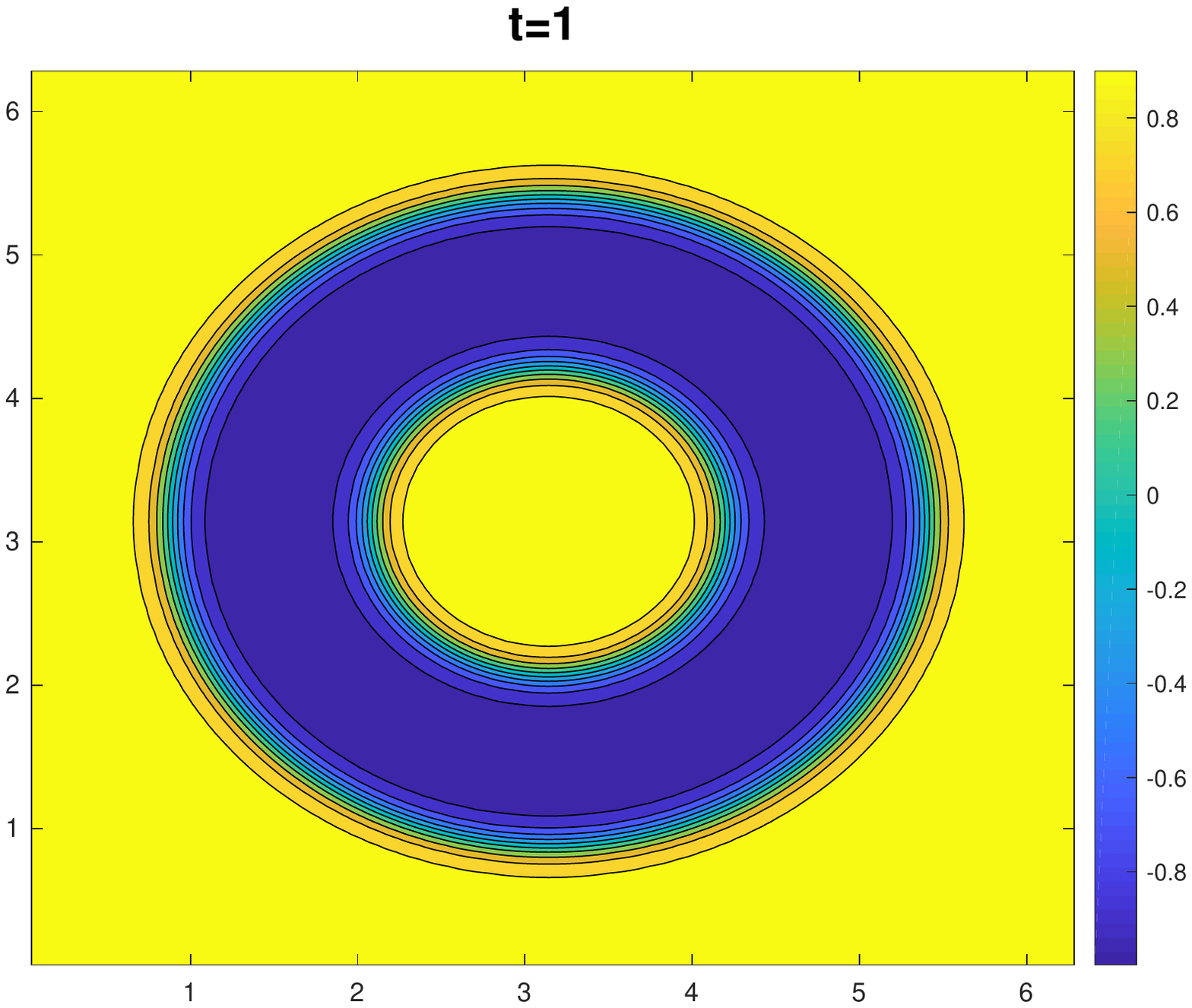}
\includegraphics[width=5cm]{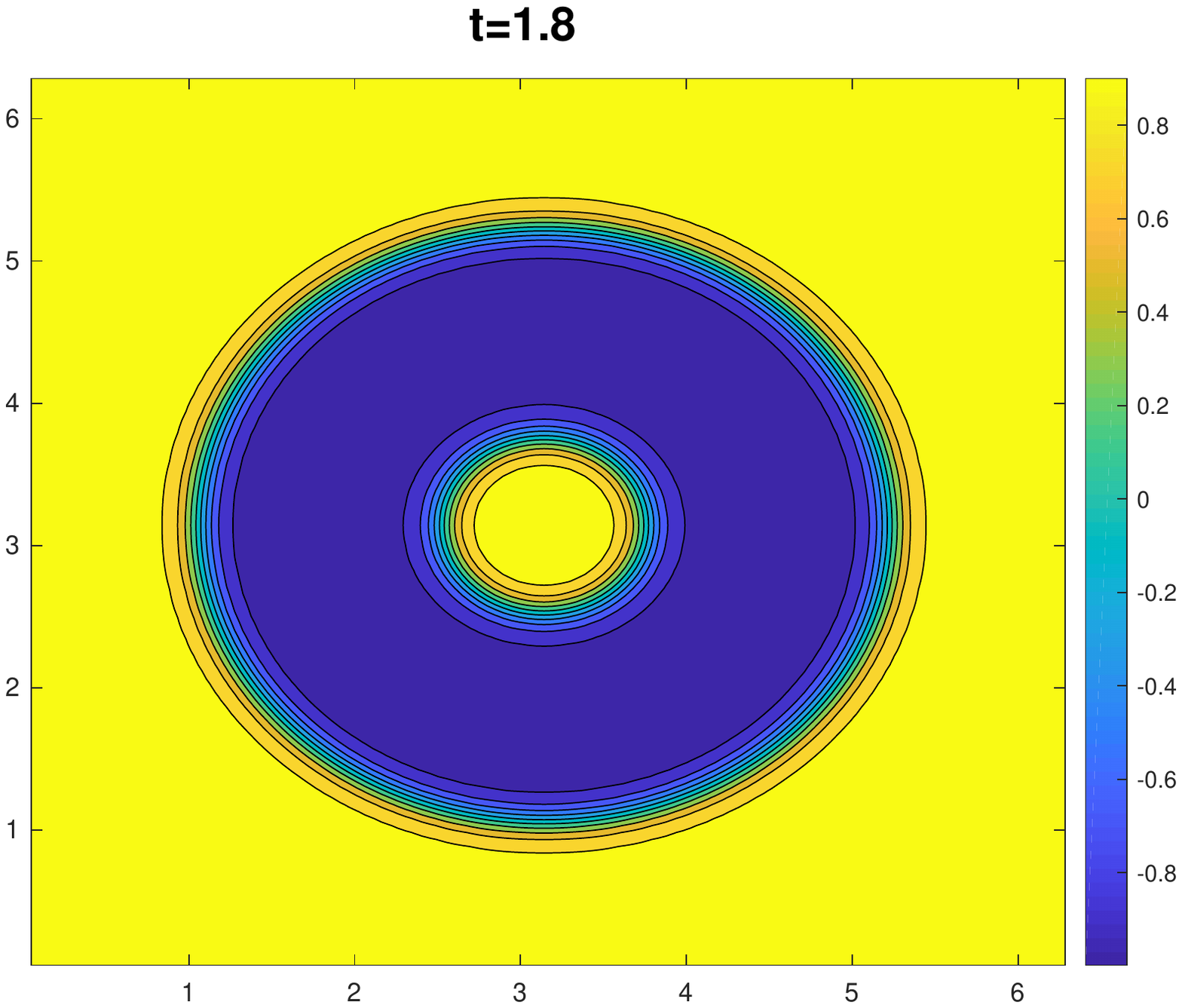}
}
\caption{Cahn Hilliard dynamics with $\epsilon = 0.1$
\label{fig:CH} }
\end{figure} 

\subsubsection{First order methods} 

Results of the numerical experiments in which $\epsilon$ is varied for the first order methods are shown in Table~\ref{t:CHfixeds}. These validate the prediction of $M = O(1/\epsilon)$ for BE and $M = O(1/\epsilon^2)$ for Eyre and IMEX1 with $\sigma$ constant. As for the AC case, SAV1 behaves similarly to IMEX1 at increased computational cost. 
For CH, the implicit problem for BE is more difficult to solve as $\epsilon \rightarrow 0$ with fixed $\sigma$, but it is still more accurate than Eyre stepping for equivalent computational cost. It will be asymptotically more efficient as long as the solution strategy for the nonlinear problem requires fewer than $O(1/\epsilon)$ iterations with adaptive time stepping. As with AC, we see that BE does not suffer from global accuracy decrease as $\epsilon \rightarrow 0$.

\begin{table}
\centerline{ \scriptsize 
\begin{tabular}{|c||c|c|c||c|c|c||c|c|} \hline 
&\multicolumn{3}{|c||}{BE}&\multicolumn{3}{|c||}{Eyre}& \multicolumn{2}{|c|}{IMEX1} \\ \hline
$\epsilon$ & $M$ & CG & $E$ & $M$ & CG & $E$ & $M$ & $E$ \\ \hline 
0.2 & 730 & 5,348 [7.33] & $*$ & 3,055 & 36,684 [12.0] & 0.019 & 
9,982 & 0.014 \\ 
0.1 & 1,184 (1.62) & 24,778 [20.9] & 0.001& 12,751 (4.17) & 190,204 [14.0] & 0.021 & 
43,332 (0.015) & 0.015 \\
0.05 & 2,068 (1.75) & 66,307 [32.1] & $*$ & 52,753  (4.13) & 937,774 [17.8] & 0.022 & 
181,234 (4.18) & 0.015 \\ 
0.025 & 3,768 (1.82) & 198,771 [52.8] & $*$ & 215,443 (4.08) & 4,504,278 [20.9] & 0.022 & 
740,366 (4.09) & 0.015 \\
\hline
\end{tabular}
}
\caption{Computational results for the first order methods applied to the CH benchmark problem with fixed local error tolerance $\sigma = 10^{-4}$ and $\epsilon$ varied. Here, $M$ is the total number of time steps taken (with the ratio to the value above in brackets), CG is the number of conjugate iterations (with the ratio to the number of time steps in brackets), and $E$ is the error in the benchmark time with $*$ denoting a result correct to three decimal places. 
\label{t:CHfixeds}
}
\end{table}

\subsubsection{Second order methods} 

The CG counts for the second order methods behave like those of BE with $\epsilon$ as shown above. We show the number of time steps used for the four methods in Table~\ref{t:CH2}, for $\sigma = 10^{-4}$ fixed and $\epsilon$ varied. The superiority of TR and BDF2 is clearly seen, consistent with $M = O(1/\epsilon)$ , compared to $M = O(1/\epsilon^{5/3})$ (noting that $2^{5/3} \approx 3.17$) for Secant, DIRK2, and SBDF2 as predicted above. Results for SAV2-A are comparable to those for SBDF2. Again, the implications for the asymptotic computational superiority of fully implicit TR and BDF2 under the assumption of sufficient solver efficiency are clear. 

\begin{table}
\centerline{
\begin{tabular}{|c|c|c|c|c|c|} \hline 
$\epsilon$ & TR & S & BDF2 & DIRK2 & SBDF2 \\ \hline 
0.2 & 230 & 534 & 320 & 378 & 1,388 \\
0.1 & 314 (1.36) &  1,530 (2.87) &  468 (1.46) & 788 (2.08) & 4,108 (2.96) \\
0.05 & 474 (1.51) & 4,722 (3.08) & 748 (1.60) & 1,906 (2.42) & 12,352 (3.01) \\ 
0.025 & 792 (1.67) & 14,924 (3.16) & 1,312 (1.75) & 6,048 (3.17) & 44,060 (3.57) \\ \hline 
\end{tabular} 
}
\caption{
Computational results for the second order methods applied to the CH benchmark problem with fixed local error tolerance $\sigma = 10^{-4}$ and $\epsilon$ varied. Shown are the total number of time steps taken (with the ratio to the value above in brackets) 
\label{t:CH2} 
}
\end{table}

\section{Asymptotic Analysis of Properties of BE AC Solutions} 
\label{s:accuracy}

The results in Table~\ref{t:fixeds} present the accuracy for BE applied to AC with fixed local error tolerance $\sigma = 10^{-4}$ under various values of $\epsilon$. It is remarkable the accuracy in the benchmark 
time does not degrade as $\epsilon \rightarrow 0$. This is unexpected, as a na\"ive prediction would be that the final accuracy scaled like $M \sigma = O(\sqrt{\sigma}/\epsilon)$ where $M$ is the number of time 
steps. It is clear that the resulting solution accuracy for the schemes under specified local error tolerance is a nontrivial question. 

We present below the asymptotic analysis of a fully implicit BE time step (\ref{eq:BE}) in two dimensions assuming the solution is in the meta-stable regime. That is,
$u_n$ is approximately described as a curve ${\bf x}_n(s)$ parametrized by arc length with normal $\hat{n}$, dressed with the heteroclinic profile (\ref{eq:asy}). We take the scaling $k = c \epsilon$ with $c$ independent of $\epsilon$, both sufficiently small depending only on the curve ${\bf x}_n$. 
We consider the formal asymptotics for the implicit time step $u$ of (\ref{eq:BE}) in this setting, anticipating that $u$ will have the same local dependence $u(s,z)$. Using 
\[
\Delta \approx \frac{1}{\epsilon^2} \frac{\partial^2}{\partial z^2} + \frac{\kappa}{\epsilon} \frac{\partial}{\partial z}
\]
where $\kappa$ is the curvature of the interface, we find at leading order $O(\epsilon^{-1})$ that $u$ has the same homoclinic structure around a new curve ${\bf x} (s)$. 
That is, 
\begin{equation}
\label{eq:unp1}
u_{n+1} \approx g(z) + \epsilon v(z,s) 
\end{equation}
with $g(z) = \tanh(z/\sqrt{2})$ and where we have changed coordinates to 
$(s,z)$ with 
\[
(x,y) = {\bf x}(s) + \epsilon z \hat{n}
\]
based on the curve  ${\bf x}(s)$ after the implicit time step. In the language of Remark~\ref{rem:profstab} we predict that the scheme preserves profile fidelity and show below that this is asymptotically consistent. 
In (\ref{eq:unp1}), $v(z,s)$ is the correction to the leading order solution. We will identify the size and structure of this term below. 

We take 
\begin{equation}
\label{eq:xspeed} 
{\bf x}_n = {\bf x} - k \rho (s) \hat{n}
\end{equation}
where $\rho$ is the average normal speed through the time step. Recalling that $k = c \epsilon$ and the spatial scaling of $z$, we have 
\begin{equation}
\label{eq:un}
u_n \approx g(z-c \rho(s)).  
\end{equation}
A diagram is shown in Figure~\ref{f:asy}. Note that the variation in normal direction appears in higher order asymptotic terms, so it is consistent in what follows to use the same $\hat{n}$ as normal direction for both curves, {\em i.e.} the same ``$z$". 
\begin{figure}
\centerline{
\includegraphics[width=8cm]{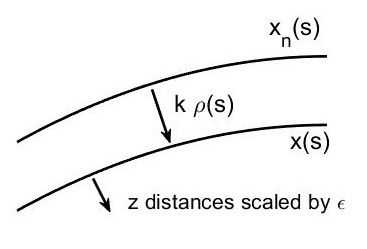}
}
\caption{Sketch of the asymptotic analysis of the fully implicit problem. Here, $\rho(s)$ the average normal speed of the interface between time steps. 
\label{f:asy}}
\end{figure}

Considering now the next order term $O(1)$ in (\ref{eq:BE}) with the forms (\ref{eq:unp1}) and ({\ref{eq:un}):
\begin{equation}
\label{eq:nextord}
g^\prime c \rho + \frac{1}{2} g^{\prime \prime} c^2 \rho^2 + \frac{1}{6} g^{\prime \prime \prime} c^3 \rho^3 \approx 
c \kappa g^\prime - c \mathcal{L} v 
\end{equation}
where $\mathcal{L} := \partial^2/\partial z^2 + f^\prime(g) \cdot$
and we have used the smallness of $c$ for the cubic Taylor approximation of $g(z) - g_n(z)$ on the right hand side. We consider (\ref{eq:nextord}) at each $s$ in the $L^2({\mathbb{R}})$ orthogonal decomposition of $G:= {\rm span}\{ g^\prime(z) \}$ and $G^\perp$. 
Note that $g^{\prime \prime} \in G^\perp$ (this does not depend on the specific reaction term $f = u^3 -u$ chosen here) and $\mathcal{L}$ has $G$ as its kernel and has bounded inverse on $G^\perp$ from standard Fredhold theory \cite{keith}. Thus we have $\rho = \kappa + O(c^2)$ and $v = O(c)$ in $G^\perp$.  Careful examination of these results shows that the errors in $G^\perp$ do not accumulate and are globally of size $O(c \epsilon) = O(k)$ and so decrease as $\epsilon \rightarrow 0$. Global errors in interface {\em position} after $O(1/k)$ time steps are of size $c^2$, independent of $\epsilon$. Global solution value errors due to the position error have size $O(c^2/\epsilon) = O(k^2/\epsilon^3)$ and so it is seen that BE behaves like a second order method in this scaling. This explains the unexpected accuracy in AC BE computations as $\epsilon \rightarrow 0$. 

\begin{remark}
\label{rem:estimator}
Note that the error estimator (\ref{eq:estimator}) uses $\mathcal{F}(u)$ which sees the undamped dominant truncation error term, which is why the number of time steps behaves with $\epsilon$ in the manner predicted in Section~\ref{s:predict}. Thus for BE applied to AC in the metastable regime, the estimator asymptotically over-estimates the local errors actually made.  
\end{remark}

The formal asymptotic results can also be used to show that the implicit time steps in this scaling lead to energy decrease. Neglecting the $O(c^2)$ terms in the interface motion, we have from (\ref{eq:xspeed})
\[
{\bf x}_n = {\bf x} - k \kappa \hat{n}.
\]
Using the identities for arc length parametrized curves $|{\bf x}_s| = 1$, $\kappa \hat{n} = {\bf x}_{ss}$ and ${\bf x}_s \cdot {\bf x}_{sss} = - \kappa^2$ it follows by taking the $s$ derivative of the equation above and the dot product with ${\bf x}_s$ at each $s$ that 
\[
| {\bf x}_{n,s} | \geq 1+ k\kappa^2 \geq{1} = |{\bf x}_s|. 
\]
This shows that the metastable curve at time $n$ is longer than at the next step $n+1$. 
Since the energy $\mathcal E$ is proportional to curve length to highest order in the metastable regime \cite{MM77}, we have shown formally that implicit time stepping for AC has the energy decay property under this time step scaling. Large, accurate, fully implicit time steps can be taken in computations validated in Section~\ref{s:comp}. 

In the next section we show a closely related rigorous result in a radial geometry. The main result is in Proposition~\ref{p:main}. A key ingredient is an identification of a dominant term in the space $G$ that represents the interface motion, separate from heavily damped terms in the perpendicular space, as shown here. Care must be taken to control the size of terms which are formally neglected in this asymptotic analysis. 

\section{Rigorous Radial Analysis of AC With BE and Eyre Time Stepping}  
\label{s:radial_main} 

We derive rigorous asymptotic evolution of a radially symmetric profile for BE and first order Eyre-type methods for
the Allen-Cahn equation in $\R^2$. Extensions to radial profiles in $\R^d$ is immediate. More precisely we consider a splitting $f=f_+-f_-$ and
study the iterative scheme
\begin{align*}
\frac{u-u_n}{k}&=u_{rr}+\frac{1}{r}u_r-\frac{1}{\varepsilon^2}(f_+(u)-f_-(u_n))\ , \hspace{0.5in}r\in[0,\infty)\\
u_r(0)&=0, u(\infty)=1.
\end{align*}
For simplicity, we assume that $f$ is smooth, odd about $u=0$, has precisely three simple zeros at $u=\pm 1$ and at $u=0$,  and tends
to $\pm\infty$ as $u\to\pm\infty$. This includes the classical choice of $f(u) = u^3-u$ but we consider other reaction terms in this class since Eyre's method can have quite different behaviour as shown in Section~\ref{s:Eyre_Type}.
The BE scheme corresponds to the choice $f_-\equiv0$ while  Eyre-type schemes take $f_+', f_-' \geq 0$. 
We pose the problem on the affine space 
\[
Y:= \{ u+1 \in H_R^1(0,\infty)\, \big|\, \partial_r u(0)=0.\},
\]
with $u_n\in Y$ as a given. The assumptions on $f$ imply that the continuous 1D Allen-Cahn equation has a steady state solution
\beq
\label{g-eq}
  g_{zz} = f(u),
\end{equation}
which is heteroclinic to $\pm1$; that is $g\to\pm 1$ as $z\to \pm\infty.$
Considering $R>1$, we modify this $g$ at the order $O(e^{-1/\vep})$ so that $g'=0$ on $(-\infty,-1/\vep)$ for $\eps\ll1$. This introduces exponentially small residuals in the sequel that have no impact upon the salient results of our analysis.

 We introduce $z=\frac{r-R}{\varepsilon}$, the weighted inner product
\[
\langle u,v\rangle_R:= \int_{-R/\vep}^\infty u(z)v(z)\, (R+\vep z)dz,
\]
and the associated spaces $L^2_R$ and $H^1_R$.  We  rewrite the iterative equation as
\beq
\label{e:It-eq}
  \frac{u-u_n}{k} = \vep^{-2} \left( u_{zz} - (f_+(u)-f_-(u_n))\right) + \frac{\vep^{-1}u_z}{R+\vep z},
 \eeq
 on the domain $z\in(-R/\vep,\infty).$   We decompose $u_n$ and $u$ as
\begin{align*} 
 u_n & = g\left(z+ \frac{R-R_n}{\vep}\right)+v_n,\\
 u&= g(z)+v,
 \end{align*}
where $R_n$ and $v_n$ are taken as given and $R$ and $v$ are to be determined. 
The profile associated to $u_n$ is denoted $g_n$ and observe that it admits the expansion
$$g_n=g\left(z+\frac{R-R_n}{\vep}\right)= g+g' \frac{R-R_n}{\vep} + O\left( \left(\frac{R-R_n}{\vep}\right)^2\right). $$ 
In the sequel we will enforce the orthogonality conditions
\begin{equation}
\label{othog-cond}
\langle v,g'\rangle_R=0, \ \ \ \ \langle v_n,g_n'\rangle_R=0,
\end{equation}
and denote the corresponding subspaces of $L^2_R$ by $X^\perp$ and $X_n^\perp$ respectively with the associated
orthogonal projections $\Pi$ and $\Pi_n$.

At this point the analysis of the implicit and Eyre-type schemes diverges sufficiently that we approach them distinctly.

\subsection{Backward Euler estimates}
\label{s:BEradial}

For  BE we take $f_-'\equiv0$, $f=f_+$, and write the iterative map as
\begin{align}
\label{BE-map}
v+\frac{k}{\vep^2}Lv =v_n-(g-g_n)+\frac{kg'}{\vep(R+\vep z)}-\frac{k}{\vep^2}\cN ,
\end{align}
where we have introduced the linear operator
\beq
\label{L-def}
L:= -\left(\partial_z^2 + \frac{ \vep}{R+\vep z} \partial_z\right) + f'(g)=-\frac{1}{R+\vep z}\partial_z\left((R+\vep z)\partial_z\right) +f'(g),
\eeq 
and the nonlinearity
\[
\cN(v):= f(g+v)-(f(g)+f'(g)v).
\]
The operator $L$ is self-adjoint in the weighted inner product
for which the eigenvalue problem takes the form
\[
  L\psi = \frac{\lambda}{R+\vep z} \psi,
\]
subject to $\partial_z\psi(-R/\vep)=0$ and $\psi\to0$ as $z\to\infty.$
Since the profile $g$ solves (\ref{g-eq}), it will be useful to compare $L$ to the simpler operator 
\beq\label{L0-def}
L_0:=-\partial_z^2 + f'(g),
\eeq
arising as the linearization of (\ref{g-eq}) about $g$ in $L^2(\R)$. The operator $L_0$ is self-adjoint on $L^2(\R)$,  
and since $g$ is heteroclinic with $g'>0$,  the Sturm-Liouville theory on $L^2(\R)$ implies that $L_0$ has a  simple, ground-state eigenvalue at $\lambda=0$ with eigenfunction 
$g'$ and the remainder of the spectrum of $L_0$ is strictly positive, in particular $L_0$ is uniformly coercive on the space $\{g'\}^\perp_{L^2(\R)}$. 
While $L$ does not generically have a kernel, it does have an eigenspace with a small associated eigenvalue. However, for $\eps$ sufficiently small, it inherits the coercivity of
$L_0$.
\begin{lemma}
\label{l:BE-coercivity}
Fix $\eps_0>0$ sufficiently small, then there exists $\alpha>0$, independent of $R\geq 1$ and of $\eps\in (0,\eps_0)$, such that
\beq
\label{BE-H1-coercivity}
\langle Lv,v\rangle_R\geq \alpha \|v\|_{H^1_R}^2.
\eeq
for all $v\in H^1_R$ satisfying $\langle v,g'\rangle_R=0.$
\end{lemma}
\begin{proof}
\noindent We defer the proof of $L^2_R$ coercivity  to the appendix.
To extend coercivity to $H^1_R$ we observe that
$$\langle Lv,v\rangle_R = \int_{-R/\vep}^\infty (R+\vep z)\left( |v'|^2+f'(g)|v|^2\right)\,dz,$$
so that for any $t\in(0,1)$ we may write
\begin{align*}
\langle Lv,v\rangle_R &= t \langle Lv,v\rangle_R+(1-t) \langle Lv,v\rangle_R, \nonumber\\
                                  &\geq \int_{-R/\vep}^\infty(R+\vep z)(t|v'|^2 + ((1-t)\alpha-t\|f'(g)\|_\infty) |v|^2)dz,\nonumber\\
                                  & \geq \tilde{\alpha}\|v\|_{H^1_R}^2,
\end{align*}
 where we have introduced $\tilde{\alpha}:= \alpha/(1+\alpha+\|f'(g)\|_\infty)>0.$
 Dropping the tilde, we have (\ref{BE-H1-coercivity}) with $\alpha$ independent of $\vep>0$ and $R>1$.
\end{proof}
\vskip 0.1in

We assume throughout our analysis that $\|v\|_{H^1_R}$ and $\|v_n\|_{H^1_R}$ are uniformly bounded by $\delta\ll 1$.
Returning to (\ref{BE-map}), we denote its right-hand side as $\cF$. To have the inversion of the operator
on the left-hand side be contractive the term $\cF$ must be approximately orthogonal to the small eigenspace of $L$. 
As Lemma\,\ref{l:BE-coercivity} shows it is sufficient to be $L_R$-orthogonal to $g'$, the kernel of $L_0.$ To this end we determine $R=\hRbe(v,v_n,R_n)$ 
such that $\cF\in X^\perp$, or equivalently
\beq
\label{BE-orthog}
\langle \cF,g'\rangle_{R}=0.
\eeq
Assuming this condition has been enforced we introduce
\[
M := I+ \frac{k}{\vep^2} L,
\]
and may rewrite the BE iteration in the equivalent formulation
\beq\label{BEC-map}
v = \cGbe(v,v_n,R-R_n):=M^{-1}\Pi \cF(v,v_n,R-R_n).
\eeq
The key step is the introduction of the operator $\Pi$, the orthogonal projection onto $X^\perp$.  This is redundant when $\cF\in X^\perp$, but preserves contractivity for choices of $(v,v_n,R_-R_n)$ when it is not.
 Our goal is to show the function $\cGbe$ is a contraction mapping and to develop asymptotic formula for $R$ and $v$. 

\begin{lemma}
The function $R=\hRbe$ satisfies the implicit relation
\begin{align}
\frac{R-R_n}{k} &= -\frac{1}{R}  +\frac{k}{4R^3} -\frac{b_1 k^2}{\vep^2R^3}+O\left(\delta,\frac{k^3}{\vep^2},\frac{\delta^2}{\vep}\right).
\label{LO_BE-iteration}
\end{align}
where
\beq
\label{BE:b1-def}
b_1:= \frac{\|g''\|_R^2}{6\|g'\|_R^2}>0.
\eeq
Moreover we have the Lipshitz estimate
\beq
\label{BE-ortho-Lip}
|\hRbe(v;v_n,R_n)-\hRbe(\tv;v_n,R_n)| \leq c \frac{k\delta}{\vep} \|v-\tv\|_{R},
\eeq
so long as $k\delta^2 \ll \vep^2.$
\end{lemma}
\begin{proof}
Due to parity considerations, we remark that $\|g'\|_R^2=R\|g'\|^2_{L^2(\R)}$, up to exponentially small terms. For brevity, and
as an element of foreshadowing, we approximate $(R-R_n)$ by $k$ in the $O$-error terms. We address the terms
in $\cF$  and derive the following elementary estimates,
\begin{align}
\label{BE:vn-est}
\langle v_n,g'\rangle_R &= \langle v_n,(g'-g_n')\rangle_{R}
 = -\langle v_n,g''\rangle_{R} \frac{R-R_n}{\vep} + O\left(\delta \frac{k^2}{\vep^2}\right),\\
\label{BE:g-gn-est}
\langle g-g_n,g'\rangle_R &=  -\|g'\|_R^2\left(\frac{ (R-R_n)}{\vep}-\frac{(R-R_n)^2}{4R\vep}\right) +\nonumber\\
   & \hspace{0.5in}+\frac{\|g''\|_R^2}{6}\frac{(R-R_n)^3}{\vep^3}+ O\left(\frac{k^4}{\vep^3}\right),\\
\label{BE:g'-est}
\left\langle \frac{g'}{R+\vep z},g'\right\rangle_{R} &=\|g'\|_{L^2(\R)}^2=\frac{\|g'\|_R^2}{R},\\
|\langle \cN,g'\rangle_{R}|&\leq c\delta^2. \label{BE:cN-est}
\end{align}
For this scheme, $\cF$ depends upon $v$ only through $\cN$. Collecting terms in the orthogonality condition that are linear
in $R-R_n$ and identifying relevant higher order terms yields the relation
\beq
\frac{R-R_n}{k} = -\frac{1}{R}  +\frac{(R-R_n)^2 }{4Rk} +\frac{b_1(R-R_n)^3}{k\vep^2}+O\left(\delta,\frac{k^3}{\vep^2},\frac{\delta^2}{\vep}\right)
\label{LO_BE-iteration2}
\eeq
where $b_1$ is given in (\ref{BE:b1-def}). Under the assumptions on $k$ and $\delta$ we have the
leading order result $ R-R_n=-k/R$. Substituting this relation into (\ref{LO_BE-iteration2}) yields
the result (\ref{LO_BE-iteration}). 

To obtain the Lipschitz estimate we observe from the estimates above that
\[
|\hRbe(v)-\hRbe(\tv)| \leq c \frac{k}{\vep\|g'\|_R^2}\left|\langle \cN(v),g'\rangle_R-\langle \cN(\tv),\tg'\rangle_R\right|.
\] 
The nonlinearity satisfies the Lipschitz properties
\[
  \|\cN(v)-\cN(\tv)\|_R\leq c \delta \|v-\tv\|_R,
 \]
 while 
 \[
 \|g'-\tg'\|_R\leq c\frac{ |\hRbe(v)-\hRbe(\tv)|}{\vep}.
 \]
 Adding and subtracting $\langle \cN(\tv),g'\rangle_R$  and using (\ref{BE:cN-est}), we arrive at the estimates
\[
 |\hRbe(v)-\hRbe(\tv)| \leq c\left( \frac{k\delta}{\vep} \|v-\tv\|_R + \frac{k\delta^2}{\vep^2}  |\hRbe(v)-\hRbe(\tv)|\right).
 \]
 Imposing the condition $k\delta^2\ll \vep^2$ yields (\ref{BE-ortho-Lip}).
\end{proof}

To establish bounds on the map $\cGbe$ defined in (\ref{BEC-map}) we apply $M$ to both sides of the relation
and take the $L^2_R$ inner product with respect to $\cGbe$. Using the coercivity  estimate
(\ref{BE-H1-coercivity}) we find
\[
\|\cGbe\|_R^2 + \alpha \frac{k}{\vep^2} \|\cGbe\|_{H^1_R}^2 \leq \|\Pi\cF\|_R\|\cGbe\|_R.
\]
Taking $v,v_n\in B_{H^1_R}(\delta)$ for  $\delta\ll1$ and recalling that
 the projection $\Pi$ crucially cancels the leading order term in $g-g_n$, we estimate
\beq
\label{BE-coercivity}
\|\cGbe\|_R + \alpha \frac{k}{\vep^2} \|\cGbe\|_{H^1_R} \leq c \left(\delta + \frac{k^2}{\vep^2}+k+ \frac{k}{\vep^2}\delta^2\right).
\eeq

For the BE system we examine distinguished limits $k=\vep^s$, for $s\in (1,2)$, which we call the large time-step regime, for which
the $H^1_R$ term is dominant on the left-hand side of (\ref{BE-coercivity}). We drop the $L^2_R$ term to find,
\[
\|\cGbe\|_{H^1_R} \leq c \left(\delta \vep^{2-s} + \vep^s +\delta^2\right).
\]
Taking $\delta=\vep^{s'}$ for any $s'>\max\{s/2,2(s-1)\}$ then we determine that
\[
\|\cGbe\|_{H^1_R}\leq c(\vep^{2-s+s'}+\vep^s+\vep^{2s'}) \leq \delta,
\]
for $\vep$ sufficiently small. In particular, since $s>\max\{s/2,2(s-1)\}$ in the large time-stepping regime, we may take $\delta=k=\vep^s$, so that,
viewing $\cGbe$ as a map on $(v,v_n)$, we have 
$\cGbe:B_{H^1_R}(k)\times B_{H^1_R}(k)\mapsto  B_{H^1_R}(k),$
for all $s$ in the large time-step regime.

\begin{prop}
\label{p:main}
Fix $1<s<2$, then in the distinguished limit $k=\vep^s$ the function $\cGbe$ defined in (\ref{BEC-map}) with $R:=R_{n+1}=\hRbe(v;v_n,R_n)$  maps
$B_{H^1_R}(k)\times B_{H^1_R}(k) $ into $B_{H^1_R}(k)$ and is a strict contraction. In particular it has a unique solution $v\in B_{H^1_R}(k)$,
denoted by $v_{n+1}$ which satisfies
\[
 \left\| v_{n+1} - \frac{k}{R^2} L\Pi g''\right\|_{H^1_R} \leq c \frac{\vep^2}{k} \|v_n\|_R + O(k^2).
\]
In particular there exists $c>0$ such that for all $v_0\in B_{H^1_R}(ck)$ and $R_0>1$ the sequence $\{(v_n,R_n)\}_{n=1}^N$ satisfies
$v_n\in B_{H^1_R}(ck)$ while $\{R_n\}_{n=0}^N$ satisfies the backwards Euler iteration
\beq
\label{LO_BE-bigstep} 
\frac{R_{n+1}-R_n}{k}= -\frac{1}{R} - \frac{b_1 k^2}{\vep^2R^3} + O(k),
\eeq
where $b_1>0$ is given by (\ref{BE:b1-def}). Here  $N$ is the iteration number such that $R_N>1$ and $R_{N+1}<1.$
\end{prop}
\begin{proof}
We have established the mapping property. To establish the contractivity we must control the impact of $f$ upon the projection $\Pi$ through
the motion of the front $R$. We assume that $v,\tv,v_n\in B_{H^1_R}(k)$ and denote 
$R=R(v)$ and $\tR=R(\tv)$, with the associated front profiles denoted by $g$ and $\tg$. The estimate (\ref{BE-ortho-Lip}) establishes
that $\hRbe$ is Lipschitz with constant $ck\delta/\vep$, which in the the large time-step regime reduces to $c k^2/\vep$. Following the
proof of (\ref{BE-ortho-Lip}) we find that
\beq
\label{BE:cF-lip}
\|\cF(v)-\cF(\tv)\|_{H^1_R} \leq c \frac{k^2}{\vep^2} \|v-\tv\|_{H^1_R}. 
\eeq
In the large time-step regime, using (\ref{BE-H1-coercivity}) we deduce the bound
\beq
\label{MPi-est}
\|M^{-1}\Pi f\|_{H^1_R} \leq \alpha^{-1} \frac{\vep^2}{k} \|f\|_{R}
\eeq
We wish to obtain a bound on the difference of $\cGbe$ at two values of $v$:
\[
\cGbe(v,v_n)-\cGbe(\tv,v_n) =  M^{-1}\Pi\cF - \tM^{-1}\tPi\tcF.
\]
We first bound the difference
\beq
\label{BE-step1}
  \cgBE:=(M^{-1}\Pi -\tM^{-1}\tPi)\cF.
 \eeq
 The analysis is complicated by the fact that $M$ is only uniformly invertible on the range of $\Pi$. To factor these projected inverses
 we act with $M$, observing
 \beq
 \label{BE-step2}
 M\cgBE = (\Pi-M \tM^{-1}\tPi)\cF = (\Pi \tM-M)\tM^{-1}\tPi\cF +\Pi(I-\tPi)\cF,
 \eeq
 where we used that fact that $\tM\tM^{-1}\tPi = \tPi$ and hence $\tM\tM^{-1}\tPi + (I-\tPi)=I.$
 Since the right-hand side of (\ref{BE-step2}) lies in the range of $\Pi$ we may invert boundedly,
 \beq
 \label{BE-step3}
 \Pi \cgBE = M^{-1}\Pi (\tM-M)\tM^{-1}\tPi\cF + M^{-1}\Pi(I-\tPi)\cF.
 \eeq
 To recover the whole $\cgBE$ we act with $(I-\Pi)$ on (\ref{BE-step1}) obtaining
 \beq
 \label{BE-step4}
 (I-\Pi) \cgBE = -(I-\Pi)\tM^{-1}\tPi\cF= 
  -(\tPi-\Pi)\tM^{-1}\tPi \cF.
 \eeq
 Adding (\ref{BE-step4}) to (\ref{BE-step3}) yields a regularized expression that accounts for the
 shifts in the projections
 \beq
 \label{BE-step5}
 \cgBE = M^{-1}\Pi (\tM-M)\tM^{-1}\tPi\cF + M^{-1}\Pi(\Pi-\tPi)\cF + (\Pi-\tPi)\tM^{-1}\tPi \cF.
 \eeq
The operators $M^{-1}\Pi$ and $\tM^{-1}\tPi$ are bounded using (\ref{MPi-est}), while
\begin{align*}
 \|\tM-M\|_{R*} &= \frac{k}{\vep}\|f'(g)-f'(\tg)\|_{R*} \leq c \|g-\tg\|_{R*},\nonumber\\
                        & \leq c \frac{|R-\tR|}{\vep}\leq c \frac{k^2}{\vep^2} \|v-\tv\|_{H^1_R},
\end{align*}
where $\|\cdot\|_{R*}$ denotes the operator norm from $L^2_R$ into itself. The projections satisfy
\begin{align*} 
\| (\Pi-\tPi)f\|_{R} &= \|g'\langle g',f\rangle_R - \tg'\langle \tg',f\rangle\|_R,\nonumber\\
                           & \leq c \frac{k^2}{\vep^2} \|v-\tv\|_{H^1_R} \|f\|_R + c\frac{\vep^2}{k} \frac{k^2}{\vep^2} \|v-\tv\|_{H^1_R}
\end{align*}
Applying these estimates to (\ref{BE-step5}) and using (\ref{BE-coercivity}) to estimate $\Pi\cF$ we obtain
\begin{align}
\|\cgBE\|_{H^1_R} \leq &c\left( \frac{\vep^4}{k^2} \frac{k^2}{\vep^2} +\frac{\vep^2}{k}\frac{k^2}{\vep^2}\right) \|v-\tv\|_{H^1_R}\|\Pi\cF\|_{L^2}, \nonumber\\
   & \leq c\left(\vep^2 + k\right) \frac{k^2}{\vep^2}\|v-\tv\|_{H^1_R}.
 \label{BE:g-step6}
 \end{align}
Finally we write
\[
\cGbe(v,v_n)-\cGbe(\tv,v_n) = \cgBE + \tM^{-1}\tPi(\cF-\tcF),
\]
and using (\ref{BE:cF-lip}), (\ref{BE:g-step6}) estimate
\begin{align*}
\|\cGbe(v,v_n)-\cGbe(\tv,v_n)\|_{H^1_R} \leq &c \left(\frac{k^3}{\vep^2}+k\right) \|v-\tv\|_{H^1_R},
\end{align*}
which is contractive so long as $k\ll \vep^\frac{2}{3}$ which holds with the large time-step regime.

Within the large time-step regime the leading order iteration (\ref{LO_BE-iteration2}) simplifies as $k\ll k^2/\vep^2$
and the dominant correction is given by the $b_1$ term. To compare to standard notation we rewrite the regime as
$\vep^2\ll k=\delta \ll 1$ and replace the internal parameter $\delta$ with $k$, the result is the large time-step interation (\ref{LO_BE-bigstep}).
\end{proof}

\subsection{Eyre-type iterations}
\label{s:Eyre_Type}
For an Eyre iteration the map (\ref{e:It-eq}) takes the form
\begin{align}
\label{E:It-map}
v+\frac{k}{\vep^2}L_+v =v_n-(g-g_n)+\frac{kg'}{\vep(R+\vep z)}+\frac{k}{\vep^2}\left(\cR-\cN\right) ,
\end{align}
where we have introduced the Eyre linear operator
\beq
\label{L+-def}
L_+:= -\left(\partial_z^2 + \frac{ \vep}{R+\vep z} \partial_z\right) + f'_+(g)=-\frac{1}{R+\vep z}\partial_z\left((R+\vep z)\partial_z\right) +f_+'(g),
\eeq 
the explicit-term residual
\[
\cR(v,v_n):= f_-(g_n) -f_-(g)+f'_-(g_n)v_n,
\]
and the nonlinearity
\[
\cN(v,v_n):= \cN_+(v)-\cN_-(v_n),
\]
which we further decompose into implicit and explicit parts
\begin{align*}
\cN_+(v)&:= f_+(g+v)-(f_+(g)+f_+'(g)v),\\
\cN_-(v_n)&:= f_-(g_n+v_n)-(f_-(g_n)+f_-'(g_n)v_n).
\end{align*}

The operator $L_+$ is self-adjoint in the weighted inner product
for which the eigenvalue problem takes the form
\[
  L_+\psi = \frac{\lambda}{R+\vep z} \psi,
\]
subject to $\partial_z\psi(-R/\vep)=0$ and $\psi\to0$ as $z\to\infty.$
The coercivity estimate is substantially simpler than for BE as the operator $L_+$
is strictly positive without constraint.
\begin{lemma}
There exists $\alpha_+>0$, independent of $R\geq 1$, such that
\beq
\label{E: H1-coercivity}
\langle L_+v,v\rangle_R\geq \alpha_+\|v\|_{H^1_R}^2.
\eeq
for all $v\in H^1_R$. 
\end{lemma}
\begin{proof}
Since $f_+'\geq0$ the normalized ground-state eigenfunction $\psi_0$ of $L_+$, satisfies
\[
\lambda_0^+ = \langle L_+ \psi_0, \psi_0 \rangle_R = \int_{-R/\vep}^\infty \left((\partial_z\psi_0)^2+f_+'(g)\psi_0^2\right)(R+\vep z)dz >0.
\]
Since the ground-state eigenvalue is strictly positive, this establishes the $L^2_R$ coercivity of $L_+$ with $\alpha_+=\lambda_0^+$. The $H^1_R$ coercivity follows as in Lemma 1.
\end{proof}

We assume throughout our analysis that $\|v\|_{H^1_R}$ and $\|v_n\|_{H^1_R}$ are uniformly bounded by $\delta\ll 1$.
We denote the right-hand side of (\ref{E:It-map}) by $\cFE$ and
introduce
\[
M_+ := I+ \frac{k}{\vep^2} L_+,
\]
which is strictly contractive on the full space $L^2_R$, and re-write the Eyre iteration as
\beq\label{E:C-map}
v = \cGE(v,v_n,R-R_n):=M_+^{-1}\Pi \cFE(v,v_n,R-R_n).
\eeq
For the Eyre iteration the role of the projection $\Pi$ is diminished as $M_+$ is contractive without it. 
Our goal is to show the existence of a map $R=\hRE(v,v_n,R_n)$, for which 
\beq
\label{E-orthog}
\langle \cFE,g'\rangle_{R}=0,
\eeq
to establish the contractive mapping properties of $\cGE$, and to 
develop asymptotic formula for $R$ and $v$.  We do this in the long time-stepping regime, $k\gg \vep^2$, which
has no lower bound for the Eyre scheme.

\begin{lemma} Assume $k\gg \eps^2$. There exists a smooth function 
$\hRE: B_{H^1_R}(\delta)\times B_{H^1_R}(\delta)\times\R\mapsto \R$ such that the profile $g=g(z;R)$ satisfies 
(\ref{E-orthog}). The function $R=\hRE$ satisfies the implicit relation
\beq
R-R_n = -\frac{\vep^2}{c_-R}  +O\left(\vep^3,\delta\vep,\frac{\vep^4}{k} \right).
\label{ELO-iteration}
\eeq
where we have introduced the leading order Eyre time constant
\beq
\label{E-c1-def}
c_-:= \frac{\langle f_-'(g)g',g'\rangle_{R}}{\|g'\|_R^2} > 0
\eeq
when $ f_-' \not \equiv 0$.
Moreover we have the Lipshitz estimate
\beq
\label{E-ortho-Lip}
|\hRE(v;v_n,R_n)-\hRE(\tv;v_n,R_n)| \leq c \vep\delta \|v-\tv\|_{R},
\eeq
so long as $\delta\ll 1$.
\end{lemma}
\begin{proof}  
Due to parity considerations, we remark that $\|g'\|_R^2=R\|g'\|^2_{L^2(\R)}$, up to exponentially small terms. 
For brevity, and as an element of foreshadowing, we approximate $(R-R_n)$ by $\vep^2$ in the $O$-error terms. 
Addressing the terms in $\cF$ one by one, we record 
\begin{align}
\label{E:vn-est}
\langle v_n,g'\rangle_R &= \langle v_n,(g'-g_n')\rangle_{R}=O(\delta\vep)\\
\label{E:g-gn-est}
\langle g-g_n,g'\rangle_R &=  -\|g'\|_R^2\frac{ (R-R_n)}{\vep}
  + O\left(\vep^3\right),\\
\label{E:g'-est}
\left\langle \frac{g'}{R+\vep z},g'\right\rangle_{R} &=\|g'\|_{L^2(\R)}^2=\frac{\|g'\|_R^2}{R},\\
\label{E:cR-est}
\langle \cR,g'\rangle_{R}&= \frac{\langle f_-'(g)g',g'\rangle_{R}(R-R_n)}{\vep} +\langle f_-'(g)v_n,g'\rangle_R + O(\vep^2,\vep\delta),\\
|\langle \cN,g'\rangle_{R}|&\leq c\delta^2.
\label{cN-est}
\end{align}
With these reductions we can simplify the orthgonality condition, identifying terms that are linear
in $R-R_n$ and most relevant higher order terms. The result is the balance
\beq
\frac{R-R_n}{k}\left(1+\frac{c_-k}{\vep^2}\right) = -\frac{1}{R}-\frac{\langle f_-'(g)g',v_n\rangle_R}{\vep\|g'\|_R^2}  
 +O\left(\vep,\delta, \frac{\delta^2}{\vep}\right),
\label{ELO-iteration2}
\eeq
where $c_-$, introduced in (\ref{E-c1-def}) is positive since $f_-'\geq0$ by assumption. The largest terms and error terms come from 
the residual, and we kept the lower order constant on the left-hand side to emphasize that in the long time-stepping regime, the residual 
dominates the natural time-step term.  Indeed, the iteration is independent of step size, $k$, given at leading order by (\ref{ELO-iteration2}).

To obtain the Lipshitz estimate we observe from the bounds above that dependence of $\hRE$ on $v$ arises from the balance
of the linear $R-R_n$ term in the residual against the nonlinearity. Since both these terms are multiplied by $k/\vep^2$ this factor cancels
and we have the balance
\[
|\hRE(v)-\hRE(\tv)| \leq  \frac{c \vep}{\langle f_-'(g)g',g'\rangle_{R}}\left|\langle \cN(v),g'\rangle_R-\langle \cN(\tv),\tg'\rangle_R\right|.
\] 
The nonlinearity satisfies the Lipshitz properties
\[
  \|\cN(v,v_n)-\cN(\tv,v_n)\|_R\leq c \delta \|v-\tv\|_R,
 \]
 while 
 \[
 \|g'-\tg'\|_{L^2}\leq c\frac{ |\hRE(v)-\hRE(\tv)|}{\vep}.
 \]
 Adding and subtracting $\langle \cN(\tv,v_n),g'\rangle_R$  and using (\ref{cN-est}), we arrive at the estimates
\[
 |\hRE(v)-\hRE(\tv)| \leq c\left( \vep\delta \|v-\tv\|_R + \delta^2 |\hRbe(v)-\hRbe(\tv)|\right).
 \]
 Imposing the condition $\delta\ll 1 $ yields (\ref{E-ortho-Lip}).
\end{proof}

We may now establish the main result on the Eyre sequence.

\begin{prop}
There exists $c>0$ such that for any  $k\gg\vep^2$ the function $\cGE$ defined in (\ref{E:C-map}) 
with $R:=\hRE(v;v_n,R_n)$  maps $B_{H^1_R}(c\vep)\times B_{H^1_R}(c\vep) $ into $B_{H^1_R}(\cGE(0,v_n),\vep^2)$ and 
is a strict contraction, satisfying
\beq
\label{E:cGE-Lip}
\|\cGE(v,v_n)-\cGE(\tv,v_n)\|_{H^1_R} \leq c \vep^2 \|v-\tv\|_{H^1_R}.
\eeq
In particular $\cGE$ has a unique fixed point in that set, which we denote $v_{n+1}$.
Moreover, if the Eyre balance parameter 
\beq
\label{E:bal-param}
\gamma:=  \|L_+^{-1}\Pi \circ f_-'(g)g'\|_{H^1_R*} <1
\eeq
then for any $\rho\in(0,1)$ there exists $c>0$ such that for all $v_0\in B_{H^1_R}(c\vep)$ and $R_0>1$, 
the sequence $\{(v_n,R_n)\}_{n=1}^N$ satisfies $v_n\in B_{H^1_R}(c\vep)$ and
\[
\frac{R_{n+1}-R_n}{\vep^2} = -\frac{c_E}{R_{n+1}}+  O\left(\vep^{1-\rho}\right).
\]
where the Eyre number, $c_E$, is defined by
\beq
\label{Eyre-Iteration}
c_E:= \frac{\|g'\|_R^2}{\langle f_-'(g)g',g'\rangle_R+ \langle K_+L_+^{-1}\Pi f_-'(g)g',f_-'(g)g'\rangle_R}>0,
\eeq
where $K_+>0$ is defined in (\ref{E:K-def}) and $K_+L_+^{-1}\Pi>0$ is self-adjoint.
\end{prop}
\begin{proof}
To establish the contractivity of $\cGE$ we follow the arguments for backward Euler, sketching only the differences.
We introduce
\beq
\label{E-step2a}
  \cgE:=(M_+^{-1}\Pi -\tM_+^{-1}\tPi)\cFE;
 \eeq
 and derive the expression
\begin{align}
\cgE  &= M_+^{-1}\Pi (\tM_+-M_+)\tM_+^{-1}\tPi\cFE + \nonumber\\
                  &\hspace{0.5in} M_+^{-1}\Pi(\Pi-\tPi)\cFE + (\Pi-\tPi)\tM_+^{-1}\tPi \cFE.
 \label{E-step2b}
 \end{align}
The operators $M_+^{-1}\Pi$ and $\tM_+^{-1}\tPi$ are bounded as $L_+$ has no small eigenvalues.
Using (\ref{E-ortho-Lip}) we estimate
\begin{align*}
 \|\tM_+-M_+\|_{R*} &= \frac{k}{\vep}\|f_+'(g)-f'_+(\tg)\|_{R*} \leq c \|g-\tg\|_{R*},\nonumber\\
                        & \leq c\frac{|R-\tR|}{\vep}\leq c \delta \|v-\tv\|_{H^1_R}.
\end{align*}
Similarly the projections satisfy
\begin{align*} 
\| (\Pi-\tPi)\|_{R*} &= \|g'\langle g',\cdot \rangle_R - \tg'\langle \tg',\cdot \rangle\|_R,\nonumber\\
                           & \leq c \delta \|v-\tv\|_{H^1_R}
\end{align*}
Applying these estimates to (\ref{E-step2a},\ref{E-step2b}) and following the proof of (\ref{E-ortho-Lip}) to estimate $\Pi\cFE$ we obtain
\begin{align}
\|\cgE\|_{H^1_R} \leq &c\frac{\vep^2}{k}\delta \|v-\tv\|_{H^1_R}\|\Pi\cF\|_{L^2},\nonumber \\
   & \leq c\left(\frac{\vep^2\delta^2}{k} +\frac{\vep^4\delta}{k}+\vep\delta + \delta^2\right)\|v-\tv\|_{H^1_R}.
 \label{E:g-step2}
 \end{align}
Finally we write
\beq
\label{E:cG-master}
\cGE(v,v_n)-\cGE(\tv,v_n) = \cgE + \tM^{-1}\tPi(\cFE-\tcFE),
\eeq
and estimate the $\cFE$ term from which the dominant contribution comes from the residual
\begin{align*}
\|\cFE- \tcFE\|_R &\leq c\frac{k}{\vep^2}\|f_-(g)-f_-(\tg)\|_R \leq c\frac{k}{\vep^2} \frac{|R-\tR|}{\vep}, \nonumber\\
& \leq c \frac{k\delta}{\vep}\|v-\tv\|_R,
\end{align*}
where we used (\ref{E-ortho-Lip}) in the last inequality. In particular we deduce that
\beq
\label{E:g-step3}
  \|\tM_+^{-1}\tPi(\cFE- \tcFE)\|_{H^1_R} \leq c  \vep\delta \|v-\tv\|_R.
 \eeq
Combining (\ref{E:g-step2}), (\ref{E:g-step3}) and (\ref{E:cG-master}), imposing $\delta=\vep$, and using $k\gg\vep^2$ we arrive at
strict contractivity on $B_{H^1_R}(c\delta)$ for any fixed $c>0$.

To establish bounds on the the fixed point $v_{n+1}$ of $\cGE(\cdot; v_n)$ we observe from (\ref{E: H1-coercivity}) that in 
the large time-stepping regime
$$\|M_+^{-1}\Pi\|_{H^1_R*} \leq c \frac{\vep^2}{k}.$$ 
Using this result we expand
\begin{align*}
\Pi\cFE &= \frac{k}{\vep^2}\Pi\left( f_-'(g)g'\frac{R-R_n}{\vep}+f_-'(g)v_n\right)+ O\left( \delta, \vep^2,k \right).
\end{align*}
Inverting $M_+$ we find, at leading order
\[
v_{n+1}=  \frac{R-R_n}{\vep} L_+^{-1}\Pi f'(g)g' + L_+^{-1}\Pi f_-'(g)v_n + O(\vep^2,\delta^2)
\]
In particular we deduce that
\[
\left \|v_{n+1} - \frac{R-R_n}{\vep} L_+^{-1}\Pi f'(g)g'\right\|_{H^1_R} \leq \gamma \|v_n\|_{H^1_R} + O(\vep^2, \delta^2).
\]
Arguing inductively, since the Eyre balance parameter $\gamma<1$ and the functions $\|L_+^{-1}\Pi f_-'(g)g'\|_{H^1_R}$ are uniformly bounded for all $R\geq 1$,
we deduce that if $\delta:=\|v_0\|_{H^1_R}=O(\vep)$ then the sequences $\{(R-R_n)\vep^{-2}\}_{0}^N$ and $\{\vep^{-1}\|v_n\|_{H^1_R}\}_0^N$ 
are uniformly bounded, independent of $\vep\ll1$ and $k\gg\vep^2$ for all $n\leq N$ so long as  $R_n>1$ for all $n=0, \ldots, N.$

To improve this bound we require Lipschitz estimates on the $v_n$ component of $\cGE$. To this end we find
\begin{align}
\label{E:cGE-Lip2}
  \|\cGE(v;v_n)-\cGE(v;\tv_n)\|_{H^1_R} &\leq \|M_+^{-1} \Pi (I+\frac{k}{\vep^2} f_-'(g))\|_{H^1_R*}\|v_n-\tv_n\|_{H^1_R}, \nonumber\\
    &\leq \left(\gamma+O\left(\frac{\vep^2}{k}\right)\right) \|v_n-\tv_n\|_{H^1_R}.
\end{align}
Here we introduce the quasi-steady parameter $\rho\in(0,1).$ Since $|R_n-R_{m}|=O(\vep^{2-\rho})$ for $|n-m|\leq N_\rho:=\leq \vep^{-\rho}$ we infer that 
\[
\left\|L_{+,n}^{-1}\Pi_n f'(g_n)f'_n - L_{+,m}^{-1}\Pi_{m} f_-'(g_{m})g'_{m}\right\|_{H^1_R} \leq c \sqrt{\vep},
\]
for all such $n$ and $m$. For $n>N_\rho$ we define the quasi-equilibrium
\[
  v_{n*} := \frac{R_{n}-R_{n-1}}{\vep} E_n(z) 
 \]
 where $E_n$ is the $R=R_n$ translate of
\beq
\label{E:E-def}
E:=  K_{+}L_{+}^{-1}\Pi f_-'(g) g'.
 \eeq
 Here the self-adjoint operator
 \beq
 \label{E:K-def}
 K_{+}:= \left(I-L_{+}\Pi \circ f_-'(g)\right)^{-1}>0,
 \eeq
 is well defined since $\|L_{+}\Pi\circ f_-'(g)\|_{H^1_R*}=\gamma<1$ by assumption. Using the Lipschitz property (\ref{E:cGE-Lip2})
  of $\cGE$ and the quasi equilibrium relation
  \[
  \| v_{n*} - \cGE(v_{n*};v_{n*})\|_{H^1_R}=O(\vep^2),
  \]
  we deduce that
 \[
 \| v_{k+1}-v_{n*}\|_{H^1_R} \leq \gamma \|v_k-v_{n*}\|_{H^1_R} + O\left(\vep^{2-\rho}\right).
 \]
 for  $k=n-m_\vep, \ldots, n$.
 Since  $\gamma^{N_\rho}\ll \vep$  we deduce from an inductive argument that
\[
 \left\|v_{n}- \frac{R_n-R_{n-1}}{\vep} E_n \right\|_{H^1_R} = O(\vep^{2-\rho}),
\]
for all $n>N_s.$ Inserting this result in (\ref{ELO-iteration}) we arrive at the leading order Eyre iteration (\ref{Eyre-Iteration}). 
\end{proof}

\begin{remark} 
\label{rem:Kpredict}
There are two examples of particular relevance
\[
\label{f1-ex}
  f(u) =u^3-u,
 \]
 with the decomposition $f_+=(1+\beta) u^3$ and $f_-=u+\beta u^3$ for $\beta>0$.
 The choice $\beta=0$  is classical and very degenerate, as in this case $f_-'(u)=1$ and the corresponding Eyre 
 balance parameter $\gamma$, defined in (\ref{E:bal-param}) is zero, and the Eyre number,  (\ref{Eyre-Iteration}) is 1. 
 In this case it is possible to rewrite Eyre's method as backward Euler with a rescaled time. In particular the slow convergence
 to equilibrium will not be in evidence.  For larger values of $\beta$ the balance parameter increases from zero and
 the Eyre number decreases from 1. As the balance parameter increases  through $1$ we anticipate enhanced slowing of the front profile 
 as the Eyre number tends to zero. The choice of non-zero $\beta$ can be viewed as spurious, a deliberate attempt to foul the method.
 A more robust example of non-zero balance arises naturally through the model
 \[
 \label{f2-ex}
 f(u)=u^5-\beta u^3,
 \]
 with $\beta\geq 1.$ This suggests the optimal decomposition $f_+=u^5$ and $f_-=\beta u^3$. Here, unambiguously, increasing
 $\beta$ increases the balance parameter and will lead to non-trivial enhanced slow-down with potential instability as $\gamma$ increases
 through $1.$
These analytic predictions are validated in a computational study below.  
\end{remark}

\begin{remark}
To leading order, in the large time-stepping regime $k\gg\vep^2$, the Eyre iteration recovers backward Euler with the substitution
$k\mapsto c_E\vep^2.$ This reduces to the exact result for the case $f(u)=u^3-u$ and $f_-(u)=u$,
for which $f_-'=1$, as the Eyre constant reduces to $1$ since $\Pi f_-(g)g'=\Pi g' =0.$

The strong contractivity of $\cGE$ with respect to $v$, given in (\ref{E:cGE-Lip}), arises from the strong convexity with respect
to $v$, but the slow evolution and marginal convergence to the quasi-equilibrium, given in (\ref{E:cGE-Lip2}) arises from the balance
between the implicit and explicit terms. The parameter $\gamma$ measures this balance,  with the quasi-equilibrium structure lost as $\gamma$ increases towards 1.
Indeed, since $\|K_+\|_{H^1_R*}\sim (1-\gamma)^{-1}$, the Eyre constant will generically tend to zero as $\gamma\to 1$.
\end{remark}

\subsubsection{Computational Validation of Remark~\ref{rem:Kpredict}}
\label{s:Epoor}

We perform computations for AC with the non-classical $f(u) = u^5 - u^3$ (which also leads to meta-stable dynamics of curvature motion) using the same initial conditions and accuracy criteria as described in Section~\ref{s:ACcomp}. BE performs almost identically to the results shown in Tables~\ref{t:fixede} and~\ref{t:fixeds} for the classical $f(u) = u^3 -u$ in terms of accuracy and variation of time steps with 
$\epsilon$ and $\sigma$. This matches the theory in Section~\ref{s:BEradial} which can be summarized as BE has profile fidelity when $k = o(\epsilon)$. 

When Eyre's method is applied to the dynamics with $f(u) = u^5 - u^3$, with the natural splitting suggested in Remark~\ref{rem:Kpredict}, profile fidelity is lost as predicted. The formal prediction of $k = O(\epsilon^{3/2})$ which was seen computationally for $f(u) = u^3 - u$ in Table~\ref{t:fixeds} is not observed for $f(u) = u^5 - u^3$. Rather, we see  $k = O(\epsilon^2)$ as predicted by the theory in the previous section. The numerical results are shown in Table~\ref{t:fixeds2}.

\begin{table}
\centerline{ 
\begin{tabular}{|c||c|c|} \hline 
&\multicolumn{2}{|c|}{Eyre with $f(u) = u^5-u^3$}\\ \hline
$\epsilon$ & $M$ & $E$ \\ \hline 
0.2  & 5,726 &  0.001 \\ 
0.1  &  21,947 (3.83) & 0.005 \\
0.05 & 86,499 (3.94) & 0.007 \\ 
0.025 &  343,525(3.97)  & 0.007 \\
\hline
\end{tabular}
}
\caption{Computational results for the AC benchmark problem with fixed local error tolerance $\sigma = 10^{-4}$ and $\epsilon$ varied, using Eyre's method with reaction term $f(u) = u^5-u^3$. Here, $M$ is the total number of time steps taken (with the ratio to the value above in brackets) and $E$ is the error in the benchmark time. 
\label{t:fixeds2}
}
\end{table}

\section{Summary and Future Work} 
\label{s:future}

We have identified the time step scaling for several first and second order schemes for AC and CH under the restriction of fixed local truncation error, $\sigma$. In particular, we derive the asymptotic behaviour of time-step number with $\sigma$ and asymptotic parameter $\epsilon$ during meta-stable dynamics.
These predictions are made under the assumption that the time steps preserve the asymptotic structure of the diffuse interface, a concept we refer to as {\em profile fidelity}.
The predictions are verified in numerical experiments.  We see that methods whose dominant local truncation error can be expressed as a pure time derivative have optimal  asymptotic performance in this particular limit. BE, TR, and BDF2 all have this desirable property. We believe these methods will also have superior performance for other problems with metastable dynamics. Our numerical results show that BE performs better than expected and we have shown an explanation of this behaviour with formal asymptotics. 

The optimal fully implicit methods asymptotically computationally outperform all linearly implicit methods in the limit we consider. We present precise criteria on the computational cost of nonlinear solvers for this comparison. The provably energy stable first and second order SAV schemes had higher computational cost than standard IMEX methods for similar results. As a final result, we present a rigorous proof that large time steps with fully implicit BE can be taken with locally unique solutions that are energy stable. This is done for the 2D radial AC equation in meta-stable dynamics. Eyre-type iteration is also considered in this analytic framework, and it is shown that in general this approach loses profile stability unless very small time steps are taken.

Our work gives strong evidence that some fully implicit schemes for phase field models should be given more consideration and that provably energy stable schemes are inaccurate and give no benefits during meta-stable dynamics. Perhaps a hybrid scheme that switches between the two approaches depending on the dynamic regime should be considered. 

Extending the analysis to the non-radial case and to CH is an interesting question. We observed that the question of global accuracy is not trivial in Section~\ref{s:accuracy} and should be considered for other schemes.  Accurate local error estimation for these problems is another interesting question to pursue. 

\section*{Appendix: $L_R^2$ coercivity of $L$}

Here we show the technical argument for Lemma~\ref{l:BE-coercivity}. 
For $u,v\in H^1(\R)$ we define the inner product
$$\langle u, v\rangle_\ell := \int_{-\ell}^\ell u(s)v(s)\,ds,$$
with the standard norms $L^2_\ell$ and $H^1_\ell$ while $L^2_{\ell^c}$ is defined in $\R\backslash[-\ell,\ell].$
Let $L_0$ be as defined in (\ref{L0-def}). 
\begin{lemma}
\label{l:A1}
Fix $\ell_0>0$ sufficiently large there exists $\alpha>0$ such that  for all $\ell>\ell_0$
\beq
\langle L_0 u,u\rangle_\ell \geq \alpha ,
\eeq
for all $u\in H^1_\ell\cap L^2(\R)$ satisfying $\langle u,g'\rangle_{L^2(\R)}=0$ and $1=\|u\|_{L^2_\ell}\geq \|u\|_{L^2_{\ell^c}}.$ 
\end{lemma}
\begin{proof}
Let $\phi$ be the minimizer of $\langle L_0 u,u\rangle_\ell$ over $H^1(\R)$ subject to $\|u\|_{L^2_\ell}=1$ and the full-line
orthogonality $\langle u,\psi_0\rangle_{L^2(\R)}=0$. By scaling, the minima is attained with  $\|\phi\|_{L^2_\ell}=\frac12$ and satisfies
$$ L_0 \phi = \lambda \phi, \hspace{0.5in} \textrm{on}\, [-\ell, \ell],$$
subject to Neumann boundary conditions $\phi_x(\pm \ell)=0$, in addition to the full line orthogonality condition.  The operator
$L_0$ on the truncated domain has eigenvalues $\lambda_0^\ell<\lambda_1^\ell<...$ which are $O(e^{-d\ell})$ far away
from the eigenvalues of $L_0$ on the full line. In particular $\lambda_0^\ell$ may be negative, but the rest are uniformly positive. 
In $L^2_\ell$ we partition $\phi=\beta \psi_0^\ell + \phi^\perp$, where $\psi_0^\ell$ is the $L^2_\ell$ ground state of $L_0$
and $\phi^\perp\in L^2_\ell$ is $L^2_\ell$ orthogonal to $\psi_0^\ell.$  Then we have
\beq
\label{a:Eq1}
\langle L_0 \phi,\phi \rangle_\ell\geq \lambda_0 \beta^2+ \lambda_1^\ell \|\phi^\perp\|_{L^2_\ell}^2.
\eeq
On the other hand the orthogonality condition implies that
$$  \langle \phi, g' \rangle_\R =0 = \beta + \langle \phi,g'\rangle_{\ell^c}.$$
where the subscript $\ell^c$ denotes integration over $\R\backslash [-\ell,\ell]$ with the corresponding norms.
In particular we deduce that
$$ |\beta|\leq \|\phi\|_{L^2_{\ell^c}} \|g' \|_{L^2_{\ell^c}}\leq \|g' \|_{L^2_{\ell^c}}\|\phi\|_{L^2_\ell}.$$
Since $g'$ decays exponentially at $\pm\infty$,  is complementary norm is exponentially small in $\ell.$
From orthogonality of $\psi_0^\ell$ and $\phi^\perp$ we have
$$ \|\phi\|_{L^2_\ell}^2 = \beta^2 + \|\phi^\perp\|_{L^2_\ell}^2\leq  \|\phi^\perp\|_{L^2_\ell}^2+\|g' \|_{L^2_{\ell^c}}^2\|\phi\|_{L^2_\ell}^2 .$$
or equivalently
$$ 1= \|\phi\|_{L^2_\ell}^2 \leq \frac{1}{1-\|g' \|_{L^2_{\ell^c}}^2}  \|\phi^\perp\|_{L^2_\ell}^2,$$
and taking $\ell$ large enough we use these bound in (\ref{a:Eq1}) to show that $\alpha$ is exponentially close to $\lambda_1^\ell>0.$
\end{proof}

To complete the proof of Lemma\,\ref{l:BE-coercivity} we take $\ell$ sufficiently large to
apply Lemma\,\ref{l:A1} and then take $\eps$ sufficiently small that $\eps |z|\leq \eps \ell\ll1$. 
Under these conditions $L^2_\ell$ and $L^2_R(-\ell, \ell)$ are equivalent norms, uniformly in $\eps$, and we have uniform 
$L^2_R(-\ell, \ell)$ coercivity of $L$.
Conversely, $L$ is clearly $L^2_R$ coercive on $[-R/\eps, \infty)\backslash[-\ell,\ell]$ since $f'(g)$ is strictly positive there.
Clearly $L$ is uniformly coercive on function with more  than half their $L^2_R$ mass in $[-R/\eps, \infty)\backslash[-\ell,\ell]$. 
The $g'$ orthogonality condition implies approximate  orthogonality to $\psi_0$ for $\ell$ large.  From these we 
deduce the full $L^2_R$-coercivity of $L$ over $X$.

\section*{Acknowledgements}
KP recognizes support from the NSF DMS under award 1813203. BW acknowledges support from an NSERC Canada grant. XC is supported by a UBC International Doctoral Fellowship. 

\bibliographystyle{abbrv}
\bibliography{PT_Bib}

\end{document}